\documentclass[preprint,12pt,number, nopreprintline]{elsarticle}

\usepackage{multirow}
\usepackage{tikz-cd}
\usepackage{subcaption}
\usepackage{amssymb}
\usepackage[ruled]{algorithm2e}
\usepackage{Macros/mydef}
\usepackage[colorlinks=true, linkcolor=blue, citecolor=blue, urlcolor=blue]{hyperref}
\newtheorem{proposition}{Proposition}
\newproof{proof}{Proof}
\newdefinition{remark}{Remark}
\usepackage{amssymb}
\usepackage{amsmath}
\usepackage{bbm}
\usepackage{fancyhdr}
\usepackage[scr=boondox]{mathalfa}
\usepackage{soul}
\usepackage[colorinlistoftodos]{todonotes}

\usepackage{comment}
\newcommand{\jw}[1]{{\leavevmode\color{black}Junjie: #1}}

\definecolor{ao}{rgb}{0.59, 0.0, 0.09}
\newcommand{\mn}[1]{{\leavevmode\color{black}#1}}

\begin{document}

\begin{frontmatter}

\title{A structure-preserving finite element framework for the Vlasov--Maxwell system}

\author[1]{Katharina Kormann}
\ead{k.kormann@rub.de}
\author[2]{Murtazo Nazarov}
\ead{murtazo.nazarov@uu.se}
\author[2]{Junjie Wen\corref{cor1}}
\ead{junjie.wen@it.uu.se}

\cortext[cor1]{Corresponding author}
\address[1]{Department of Mathematics, Ruhr-Universität Bochum, Universitätsstraße 150, D-44801 Bochum, Germany}
\address[2]{Division of Scientific Computing, Department of Information Technology, Uppsala University, Uppsala 751 05 Sweden.}

\begin{abstract}
We present a stabilized, structure-preserving finite element framework for solving the Vlasov--Maxwell equations. The method uses a tensor product of continuous polynomial spaces for the spatial and velocity domains, respectively, to discretize the Vlasov equation, combined with curl- and divergence-conforming N\'ed\'elec and Raviart-Thomas elements for Maxwell's equations on Cartesian grids. A novel, robust, consistent, and high-order accurate residual-based artificial viscosity method is introduced for stabilizing the Vlasov equations. The proposed method is tested on the 1D2V and 2D2V reduced Vlasov--Maxwell system, achieving optimal convergence orders for all polynomial spaces considered in this study. Several challenging benchmarks are solved to validate the effectiveness of the proposed method.
\end{abstract}




\end{frontmatter}

\section{Introduction}
This paper focuses on the numerical approximation of solutions to the Vlasov--Maxwell system. This system models the dynamics of ionized gas consisting of charged particles, or plasma, under the influence of electromagnetic fields. Numerous fields in science and engineering rely on accurate approximations of the Vlasov--Maxwell system. These include modeling plasmas in stars, the solar wind, and magnetospheres. Another application that is more oriented towards engineering is fusion energy. The solutions of the Vlasov--Maxwell system play a crucial role in controlling the confinement in fusion reactors such as tokamaks and stellarators.

The Vlasov--Maxwell system consists of a first-order nonlinear hyperbolic equation for the distribution function of plasma and Maxwell's equations for the self-consistent electric and magnetic fields of the plasma. Analytical solutions to these equations are generally unavailable. Therefore, numerical approximations of these solutions are necessary to study the physical phenomena they describe. Simulating such a system using numerical schemes is a challenging task. On the one hand, numerical schemes often face difficulties such as high computational costs and numerical instabilities. On the other hand, the numerical solutions are expected to maintain the important physical properties of the system, such as conservation (for instance, mass, energy, and charge conservation), stability (the numerical scheme must remain stable over long-time simulations), and divergence-conforming properties. Specifically, the numerical solution must satisfy Gauss' law, ensuring $\GRAD_{\textcolor{black}{\bx}} \SCAL \bB = 0$ and $\GRAD_{\textcolor{black}{\bx}} \SCAL \bE = \rho/\epsilon_0$, where $\bB$ is the magnetic field, $\bE$ is the electric field, $\epsilon_0$ is the vacuum permittivity, and $\rho$ is the charge density, at every point in space and at every time level, either strongly or weakly.

The particle-in-cell (PIC) method is popularly used for numerical simulation of plasma physics, including the Vlasov--Maxwell equations, due to its low computational cost when applied to high-dimensional problems and its ease of implementation. This method represents the particle distribution using so-called macroparticles and is coupled with appropriate Maxwell solvers for describing the electromagnetic fields. Many robust schemes for Maxwell's equations are considered in this field, aiming to preserve the essential properties of the Vlasov--Maxwell equations, see, e.g., \cite{10.1063/1.4935904,Qin:2015hta,Kraus_Kormann_Morrison_Sonnendrücker_2017, MR4165619,Bailo_Carrillo_Hu_2024} and references therein. While PIC methods are widely used for plasma simulations, they have some limitations. For instance, due to their construction, PIC methods produce numerical noise, which reduces accuracy in capturing small-scale features of the solution. 

Due to the above-mentioned limitations of the PIC method, grid-based methods are becoming an active research direction for simulating the Vlasov--Maxwell system. Galerkin approximations, as grid-based methods, of hyperbolic systems are gaining popularity lately. The advantages of this approach include the ability to handle complex geometries and the ease of achieving higher-order accuracy. We refer the reader to \cite{MR3267101,Juno_2018,Munch_2021} and references therein for recent work on Discontinuous Galerkin (DG) methods for plasma simulations. However, handling repeated degrees of freedom at the edges of elements, designing robust high-order stabilization techniques across edges, addressing the severe time step restriction for higher-order polynomials as pointed out in \cite{Weber_2022}, and ensuring stable divergence-conforming polynomial spaces present significant challenges for DG schemes, particularly in high-dimensional problems such as the Vlasov--Maxwell system.

The finite element method (FEM), a class of Galerkin methods, offers notable flexibility in handling complex geometries and supports high-order accuracy through enrichment of the polynomial approximation space. Moreover, FEM allows the construction of divergence-conforming function spaces, which is particularly advantageous for approximating Maxwell’s equations. However, generating high-dimensional meshes and corresponding function spaces remains a significant challenge for both FEM and DG methods.

One promising approach to address this issue is using a tensor product of two lower-dimensional spaces, which facilitates the extension of numerical schemes to high-dimensional problems, as explored in \citep{Munch_2021} and \citep{LOVELAND2022101831}. In this project, we aim to leverage a tensor-product of two finite elements of the spatial and the velocity domain, respectively, to solve the Vlasov–Maxwell system in higher dimensions, focusing initially on at least the 2D2V case.

One of the main drawbacks of the finite element method (FEM) for solving the Vlasov equation is its sensibility to numerical instabilities. To address this, additional stabilization mechanisms are necessary to ensure the robustness of the numerical schemes. At the time of writing, only a limited number of studies have explored this direction for high-dimensional problems. For instance, in \cite{Asadzadeh_2018}, the authors present a posteriori error analysis of a Galerkin–Least-Squares (GLS) stabilized FEM scheme for the 1D1V Vlasov equation. GLS-based stabilization presents substantial challenges in achieving high-order, robust, and efficient schemes in high-dimensional settings.

Another stabilization strategy is viscous regularization, such as {\em residual-based} artificial viscosity, which introduces diffusion terms into the advection equations. In our earlier work \citep{wen2025anisotropicnonlinearstabilizationfinite}, we successfully applied this approach to the 1D1V Vlasov–Poisson system, attaining stable, high-order approximations. Commonly, the residual-based artificial viscosity techniques rely on computing residuals via an $L^2$-projection, as done in \citep{wen2025anisotropicnonlinearstabilizationfinite, MR4208958, MR3998292}. However, this projection becomes computationally expensive in high-dimensional problems.

In this work, we propose a new framework for the viscous regularization of the Vlasov–Maxwell system. The aim is to stabilize the Vlasov equation while simultaneously preserving Gauss' laws. Our approach provides an efficient and novel method for computing residuals and applying artificial viscosity in high-dimensional settings such as the Vlasov equation. It is designed to integrate seamlessly with structure-preserving, curl- and divergence-conforming finite element spaces for solving Maxwell’s equations. Unlike GLS stabilization, our method employs a nonlinear anisotropic shock-capturing technique that introduces a consistent, nonlinear viscous regularization. Key advantages include compatibility with high-order explicit time integration and robustness, conservativeness, and stability—regardless of the underlying polynomial degree.

This manuscript is organized as follows: All the equations in the system of Vlasov--Maxwell, together with the structure of Maxwell's equations, are introduced in Section \ref{sec:prelim}. In Section \ref{sec:visc:reg}, we propose a viscous regularization method at the PDE level, and then in Section~\ref{sec:FEVM}, we present the anisotropic stabilization method used in this paper. In the same section, we outline a structure-preserving finite element framework for Vlasov--Maxwell equations. The numerical results of the experiments are presented in Section \ref{sec:exp}.

\section{Preliminaries}\label{sec:prelim}
In this section, we introduce the governing equations of Vlasov--Maxwell. We also propose the viscous regularization for the Vlasov--Maxwell equation and define the notations for the finite element methods.
\subsection{Governing equations} \label{Sec:governing_equations}
For a particle species $s$ with charge $q_s$ and mass $m_s$, we denote its distribution function in the phase space by $f_s({\pmb{x}} , {\pmb{v}} , t)$. Let the spatial and velocity space be denoted by $\Omega_{{\pmb{x}} }$ and $\Omega_{{\pmb{v}} }$, respectively, $\Omega_{{\pmb{x}} } \subseteq \mathbb{R}^{d_{\pmb{x}}}$ and $\Omega_{{\pmb{v}} } \subseteq \mathbb{R}^{d_{\pmb{v}}}$, where $d_{\pmb{x}}$ and $d_{\pmb{v}}$ represent the dimensions of the two spaces. The dimensions are $d_{\pmb{x}}=d_{\pmb{v}}=3$ but reduced models, like $d_{\pmb{x}}=1,d_{\pmb{v}}=2$, can be considered under certain symmetry assumptions. In addition, the velocity space is unbounded in general but the values of the distribution function decay fast as $\|{\pmb{v}}\|_2 \rightarrow \infty$. Thus the computational domain is chosen such that the value of the distribution functions is neglectable outside and artificial boundary conditions are added. \textcolor{black}{It is common to close the velocity domain by periodic boundaries, even though being unphysical the effect of the boundary conditions is neglectable for large enought boundaries and periodic boundary conditions ensure that the system is closed. Also in configuration space, we assume boundary conditions that ensure a closed system (\ie boundary integrals vanish on integration by parts), in particular we consider periodic boundary conditions, unless stated otherwise.}

Let us define the phase space as $\Omega := \Omega_{{\pmb{x}} }\otimes\Omega_{{\pmb{v}} }$. The Vlasov equation is then given by
\begin{equation}\label{eq:vm}
\begin{aligned}
  \frac{\p  f_s }{\p  t}+{\pmb{v}} \SCAL\nabla_{{\pmb{x}} }f_s +\frac{q_s}{m_s}(\bE +{\pmb{v}} \CROSS\bB )\SCAL\nabla_{{\pmb{v}} }f_s  &\, = 0, 
  \quad ({\pmb{x}} ,{\pmb{v}} ,t)\in\Omega\CROSS\mathbb{R}^+, \\
  f_s ({\pmb{x}} ,{\pmb{v}} ,0) &\, =  f_{s,0}({\pmb{x}} ,{\pmb{v}} ),
  \quad ({\pmb{x}} ,{\pmb{v}} )\in\Omega,
\end{aligned}
\end{equation}
where $f_{s,0}({\pmb{x}} ,{\pmb{v}} )$ is a given initial condition and the self-consistent part of the electric field $\bE $ and magnetic field ${\bB }$ are described by the following Amp\`ere's law
\begin{equation}\label{eq:AP}
  \frac{1}{c^2}\frac{\p \bE }{\p  t}=\nabla_{{\pmb{x}} }\CROSS\bB -\mu_0\bJ,
\end{equation}
and Faraday's law
\begin{equation}\label{eq:FR}
  \frac{\p \bB }{\p  t}=-\nabla_{{\pmb{x}} }\CROSS\bE,
\end{equation}
where $c = \frac{1}{\sqrt{\epsilon_0\mu_0}}$ is the speed of light and $\epsilon_0$ and $\mu_0$ the vacuum permittivity and vacuum permeability, respectively. 
In addition, the electric and magnetic fields should satisfy the constraints
\begin{equation}\label{eq:GL}
\begin{aligned}
  \nabla_{{\pmb{x}} }\SCAL\bE &=\frac{\rho}{\epsilon_0},\\
  \nabla_{{\pmb{x}} }\SCAL\bB &=0,
\end{aligned}
\end{equation}
which are commonly called Gauss' law for electricity and magnetism. The charge density $\rho$ and current density $\bJ $ are derived from the distribution function $f_s$ as follows
\begin{equation}\label{eq:denst}
  \rho =  \sum_s q_s \int_{\Omega_{{\pmb{v}} }} f_s \ {\rm d}{\pmb{v}} ,\qquad\bJ  =  \sum_s q_s \int_{\Omega_{{\pmb{v}} }}{\pmb{v}} f_s \ {\rm d}{\pmb{v}} .
\end{equation}

Maxwell's equation can also be expressed in terms of the scalar potential $\Phi$ and the vector potential $\mathbf{A}$. They are related to the fields by
\begin{equation}
  \begin{aligned}
\bE &= - \nabla_{{\pmb{x}}} \Phi - \frac{\partial \mathbf{A}}{\partial t},\\
\bB &= \nabla_{{\pmb{x}}} \CROSS \mathbf{A}.
\end{aligned}\notag
\end{equation}
The scalar potential $\Phi$ can be determined by solving the following Poisson equation
\begin{equation}
    -\nabla_{{\pmb{x}} }^2\Phi = \frac{\rho}{\epsilon_0}.\label{eq:poisson}
\end{equation}

The system of Maxwell's equations has a complex geometric structure itself. Computing the divergence of Faraday's law \eqref{eq:FR}, we get the following conclusion
\begin{equation}
    \nabla_{{\pmb{x}} }\SCAL\p _t\bB =0,\notag
\end{equation}
which implies that the divergence-free property of $\bB $ holds over time if it is satisfied at $t=0$. Similarly, we deduce that Gauss' law \eqref{eq:GL} for electricity holds if the following continuity equation for charge conservation holds
\begin{equation}\label{eq:ctcharge}
    \p _t\rho+\nabla_{{\pmb{x}} }\SCAL\bJ =0.
\end{equation}
The above equation \eqref{eq:ctcharge} is a compatibility condition for Maxwell's equations, and it is obtained by computing the integral of \eqref{eq:vm} in velocity space. 

Taking the divergence of the Amp\`ere's equation in \eqref{eq:AP} and by noting that $\GRAD_{{\pmb{x}}} \SCAL (\GRAD_{{\pmb{x}}} \times \bB)=0$, and using \eqref{eq:ctcharge} we obtain the Gauss' law for electricity in a different equivalent form:
\begin{equation}\label{eq:GL_E_dt}
\p_t\left(\GRAD_{{\pmb{x}}} \SCAL \bE - \frac{\rho}{\epsilon_0}\right) = 0,
\end{equation}
which implies that the law remains satisfied at all times, provided it holds for the initial data.

The Vlasov--Maxwell system has the following conservation properties. 
\begin{proposition}[Conservation]
The following quantities are conservative for the solution of Vlasov--Maxwell equations:
\begin{enumerate}
\item Mass:
$\frac{\rm d}{{\rm d}t}\sum_s\int_{\Omega} m_sf_s \ {\rm d}{\pmb{x}} {\rm d}{\pmb{v}} =0.$
\item Momentum:
$\frac{\rm d}{{\rm d}t}(\sum_s\int_{\Omega} m_sf_s {\pmb{v}} \ {\rm d}{\pmb{x}} {\rm d}{\pmb{v}} +\epsilon_0\int_{\Omega_{{\pmb{x}} }}\bE \CROSS\bB \ {\rm d}{\pmb{x}} )=0.$
\item Energy:
$\frac{\rm d}{{\rm d}t}\left(\sum_s\int_{\Omega} m_sf_s {\pmb{v}} ^2\ {\rm d}{\pmb{x}} {\rm d}{\pmb{v}} +\int_{\Omega_{{\pmb{x}} }}(\epsilon_0\bE ^2+\mu_0^{-1}\bB ^2)\ {\rm d}{\pmb{x}} \right)=0.$
\item Squared $L^2$-norm:
$\frac{\rm d}{{\rm d}t}\int_{\Omega}\vert f_s \vert^2\ {\rm d}{\pmb{x}} {\rm d}{\pmb{v}} =0.$
\end{enumerate}
\end{proposition}
\begin{proof}
Mass conservation is trivial, which is obtained by integrating the equation \eqref{eq:vm}  in $\Omega$. The momentum, energy, and squared $L^2$-norm conservation are obtained by multiplying \eqref{eq:vm} by ${\pmb{v}}$, ${\pmb{v}}^2$, and $f$ respectively and integrating in $\Omega$.
\end{proof}

\subsection{Viscous regularization of the Vlasov--Maxwell system}
\label{sec:visc:reg}
By defining the vector field $\pmb{\bbetaa} := \big(\bbetaa_{{\pmb{x}} }, \bbetaa_{{\pmb{v}} }\big)^\top := \big({\pmb{v}} , \frac{q_s}{m_s}(\bE  + {\pmb{v}}  \CROSS \bB ) \big)^\top$, the Vlasov equation \eqref{eq:vm} can be expressed as the following nonlinear advection equation for $({\pmb{x}} ,{\pmb{v}} ,t)\in\Omega\CROSS\mathbb{R}^+$:
\begin{equation}\label{eq:vw:advection}
\begin{aligned}
\p _t f_s + \bbetaa_{{\pmb{x}} } \SCAL \nabla_{{\pmb{x}} } f_s + \bbetaa_{{\pmb{v}} } \SCAL \nabla_{{\pmb{v}} } f_s & = 0,\\
f_s ({\pmb{x}} ,{\pmb{v}} ,0) &=  f_{s,0}({\pmb{x}} ,{\pmb{v}} ).
\end{aligned}
\end{equation}
We aim to approximate the above equation using continuous finite element methods. However, it is well-known that finite element approximations of advection problems are unstable. Therefore, additional stabilization terms must be introduced, either to the fully discrete approximation or at the PDE level. Viscous regularization is a well-studied stabilization technique commonly used in advection-dominated problems; see, for instance, \cite{Guer2014, Dao2022a}. The viscous regularized Vlasov equation reads: we look for viscosity solution $f^{\epsilon}_s({\pmb{x}}, {\pmb{v}}, t)$ from:
\begin{equation}\label{eq:vw:advection:visc}
\begin{aligned}
\p _t f^{\epsilon}_s + \bbetaa_{{\pmb{x}} } \SCAL \nabla_{{\pmb{x}} } f^{\epsilon}_s &+ \bbetaa_{{\pmb{v}} } \SCAL \nabla_{{\pmb{v}} } f^{\epsilon}_s \\
&-
\GRAD_{{\pmb{x}}}\cdot\big( \mathsf{A_{{\pmb{x}}}} \GRAD_{{\pmb{x}}} f^\epsilon_s \big)
-
\nabla_{{\pmb{v}}}\cdot \big( \mathsf{A_{{\pmb{v}}}} \GRAD_{{\pmb{v}}} f^\epsilon_s \big)
= 0, \\
&\, \hspace{1.54in}f^{\epsilon}_s ({\pmb{x}} ,{\pmb{v}} ,0) =  f_{s,0}({\pmb{x}} ,{\pmb{v}} ),
\end{aligned}
\end{equation}
where $\mathsf{A}_{{\pmb{x}}}$ and $\mathsf{A}_{{\pmb{v}}}$ are vanishing, nonnegative, and sufficiently smooth viscosity matrices that are yet to be defined. 

Since additional terms are added to the Vlasov equation, the continuity equation \eqref{eq:ctcharge} should take into account the added terms. When taking the integral of \eqref{eq:vw:advection:visc} in $\Omega_{\pmb{v}}$, we obtain the following viscous continuity equation:
\begin{equation}\label{eq:ctcharge:visc}
    \p _t\rho^\epsilon+\GRAD_{{\pmb{x}}}\SCAL
    \bJ^\epsilon -\int_{\Omega_{{\pmb{v}}}}\GRAD_{{\pmb{x}}} \SCAL \big( \mathsf{A}_{{\pmb{x}}}\GRAD_{{\pmb{x}}}f^\epsilon_s \big)\ {\rm d}{\pmb{v}}=0.
\end{equation}
Since the continuity equation includes an additional viscosity term, Gauss' law for electricity will be violated. Once the divergence of the Amp\`ere equation is calculated, the additional viscous term in \eqref{eq:ctcharge:visc} will be the remaining term. Therefore, we propose solving the following regularized Maxwell's equations:
\begin{equation}\label{eq:AP:visc}
\begin{aligned}
 \frac{1}{c^2} \frac{\p\bE^\epsilon }{\p  t}&=\nabla_{{\pmb{x}} }\CROSS\bB^\epsilon -\mu_0\bJ^\epsilon 
  + \mu_0
  \int_{\Omega_{{\pmb{v}}}}\mathsf{A}_{{\pmb{x}}}\GRAD_{{\pmb{x}}}f^\epsilon_s \, \ud{\pmb{v}},\\
  \frac{\p \bB^\epsilon }{\p  t}&=-\nabla_{{\pmb{x}} }\CROSS\bE^\epsilon.
\end{aligned}
\end{equation}
The solution of the last equation \eqref{eq:AP:visc} satisfies the Gauss' law \eqref{eq:GL}. 
To simplify the discussion, we propose using a modified charge density:
\begin{equation}\label{eq:J_h:visc}
\widetilde{\bJ^\epsilon} := \bJ^\epsilon - \int_{\Omega_{{\pmb{v}}}}\mathsf{A}_{{\pmb{x}}}\GRAD_{{\pmb{x}}}f^\epsilon_s \, \ud{\pmb{v}}.
\end{equation}
which will be used throughout the rest of the paper.

For simplicity, we omit the superscript $"^\epsilon"$ from the regularised solutions. Below, we propose a finite element approximation of the regularized system.  

\subsection{Finite element approximations}
We will consider an approximation between the domains $\Omega_{{\pmb{x}}}$ and $\Omega_{{\pmb{v}}}$ using tensor product. Let us define the mesh and function space in each domain separately.
Two types of elements are commonly used for FEM: $\polP_k$ elements, which employ polynomials on triangles and tetrahedrons, allowing flexible meshes for complex geometries; and $\polQ_k$ elements, which use tensor-product polynomials on quadrilaterals or hexahedrons. Compared to $\polP_k$ elements, $\polQ_k$ elements are generally more accurate.  In this work, we primarily use $\polQ_k$ elements, but we emphasize that there is no barrier to switching to $\polP_k$ elements if desired.

Let $\calT_{\pmb{x}}$ be a triangulation of the domain $\Omega_{\pmb{x}}$, consisting of a finite union of cuboids 
$K_{\pmb{x}} := \prod^{d_{\pmb{x}}}_{l=1}[x_{l}^-, x_{l}^+]$, where $x_{l}^-, x_{l}^+\in \mathbb{R}$ and $x_{l}^- < x_{l}^+$, such that $\overline{\Omega}_{\pmb{x}} = \bigcup_{K_{\pmb{x}}\in \calT_{{\pmb{x}}}} \overline{K}_{\pmb{x}}$, where $\overline{\Omega}_{\pmb{x}}$ and $\overline{K}_{\pmb{x}}$ denote the closure of $\Omega_{\pmb{x}}$ and $K_{\pmb{x}}$ respectively. Let us denote the mesh size of the cuboid in each direction of the $\pmb{x}$ spaces as $\Delta{x_l}:= x_{l}^+ - x_{l}^- $, $l=1,\ldots,d_{\pmb{x}}$. We construct a finite element space in $\mathbb{R}^{d_{\pmb{x}}}$ using the following $\polQ_k$ polynomial space:
\begin{equation}
	\polQ_k({\pmb{x}})={\rm span}\left\{\prod^{d_{\pmb{x}}}_{l=1}x_i^{\alpha_l}=\pmb{x}^{\pmb{\alpha}}:0\le\alpha_l\le k\ {\rm for}\ l=1,\ldots,d_{\pmb{x}}\right\},\notag
\end{equation}
where $\pmb{\alpha} = (\alpha_1, \ldots, \alpha_{d_{\pmb{x}}}) \in \mathbb{N}^{d_{\pmb{x}}}$ is the multi-index. For the mesh $\calT_{\pmb{x}}$, we define the following continuous approximation space:
\begin{equation*}
	\polV_{\pmb{x}} := \{w(\pmb{x})\in \calC^0(\overline\Omega_{\pmb{x}}):\, w(\bT_K(\widehat{\pmb{x}}))\in \polQ_{k}(\widehat{K}),\forall K\in \calT_{\pmb{x}} \},
\end{equation*}
where $\bT_K(\widehat{\pmb{x}})$ denotes an affine mapping from the reference cuboid $\widehat{K}$ to the physical cell $K\in \calT_{{\pmb{x}}}$.

We adopt the definitions above and define $\calT_{\pmb{v}}$ and $\polV_{{\pmb{v}}}$ for $\Omega_{{\pmb{v}}}$.  Let $\calI_{\pmb{x}}:=\{ 1:N_{\pmb{x}} \}$ and $\calI_{\pmb{v}}:=\{ 1: N_{\pmb{v}} \}$ be the set of global nodes of the spaces $\polV_{{\pmb{x}}}$ and $\polV_{{\pmb{v}}}$ respectively, where $N_{\pmb{x}}={\rm dim}(\polV_{{\pmb{x}}})$ and $N_{\pmb{v}}={\rm dim}(\polV_{{\pmb{v}}})$. We denote the basis set of $\polV_{{\pmb{x}}}$ by $\{\phi_i\}_{i\in \calI_{{\pmb{x}}}}$, and the basis set of $\polV_{{\pmb{v}}}$ by $\{\varphi_i\}_{i\in \calI_{{\pmb{v}}}}$. The Vlasov equation is solved in the mesh $\calT:= \calT_{{\pmb{x}}} \otimes \calT_{{\pmb{v}} }$ and the space $\polV:= \polV_{{\pmb{x}} }\otimes\polV_{{\pmb{v}}}$. 
We also denote by $\calI:=\{ 1:N \}$ the set of global nodes of $\polV$, and we know that $N:=N_{\pmb{x}} N_{\pmb{v}}$. Let $\{\psi_i\}_{i\in \calI}$ represent the collection of all Lagrange basis functions in $\polV$. We further define $S_x(j)$, $S_v(j)$, and $S(l)$ as the set of all nodes within the support of $\phi_i$, $\varphi_j$, and $\psi_l$ for all $i\in\calI_\bx$, $j\in\calI_\bv$, and $l\in\calI$, respectively.

Since $\polV:=\polV_{{\pmb{x}}}\otimes\polV_{{\pmb{v}}}$, for any $\psi_l\in\polV$, there exists a pair $(\phi_i,\varphi_j)\in(\polV_{{\pmb{x}}},\polV_{{\pmb{v}}})$ such that $\psi_l=\phi_i\varphi_j$. We let $\psi_{(i-1)N_{\pmb{v}}+j}=\phi_i\varphi_j$ for $i\in\calI_{\pmb{x}}$, $j\in\calI_{\pmb{v}}$, and we define the following subscript for the degrees of freedom of $\polV$: $\psi_{i,j}=\psi_{(i-1)N_{\pmb{v}}+j}$ for clarity of notation.
\begin{remark}
When $d_{\pmb{x}} = d_{\pmb{v}} = 1$, $\calT$ is always a grid mesh, as shown in Figure \ref{fig:mesh} (a). However, for higher-dimensional settings, more options become available, see Figure \ref{fig:mesh} (b) and (c). For instance, the $\polP_k$ element of the FEM is known for its robustness in handling complex geometries, and one can use the $\polP_k$ elements for the spatial domain, while employing $\polQ_k$ elements in the velocity space to achieve high accuracy. In addition, different polynomial degrees can also be utilized for $\polV_{{\pmb{x}}}$ and $\polV_{{\pmb{v}}}$ if necessary.
\end{remark}
\begin{figure}[htbp]
  \centering
  \begin{subfigure}[b]{.33\textwidth}
    \includegraphics[width=\linewidth]{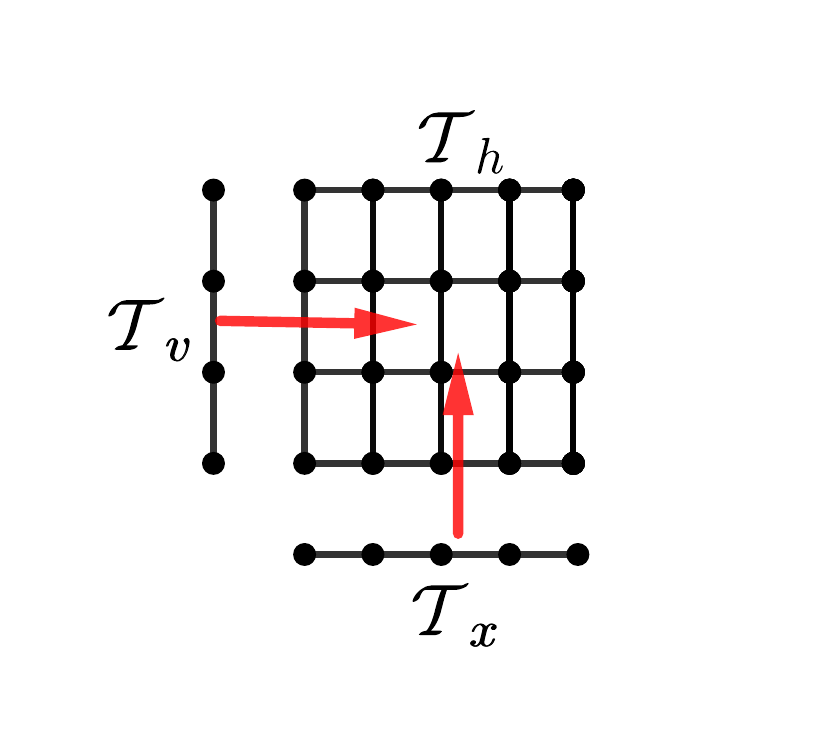}
    \caption{${\rm 1D1V}$}
  \end{subfigure}\hfill
  \centering
  \begin{subfigure}[b]{.33\textwidth}
    \includegraphics[width=\linewidth]{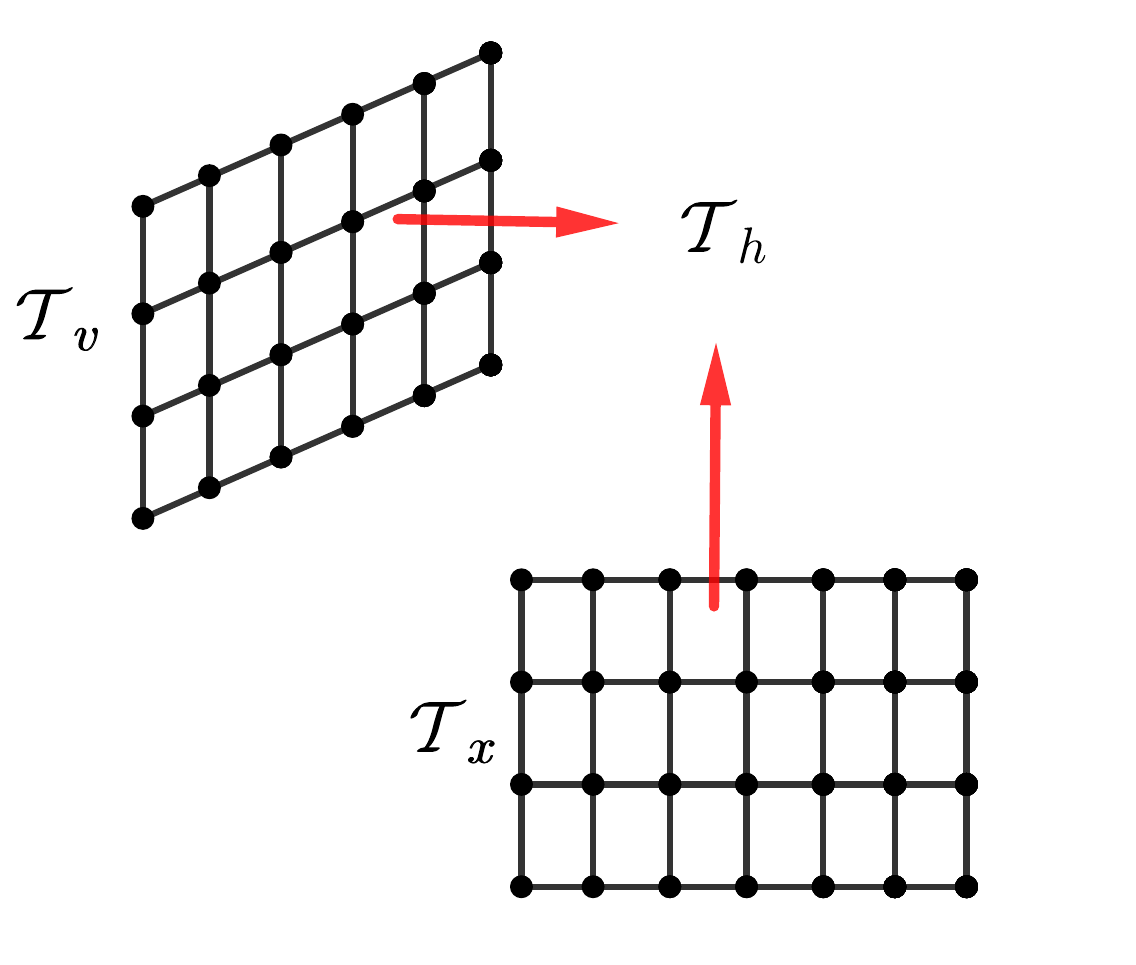}
    \caption{${\rm 2D2V},\polQ_k\otimes\polQ_k$}
  \end{subfigure}\hfill
  \centering
  \begin{subfigure}[b]{.33\textwidth}
    \includegraphics[width=\linewidth]{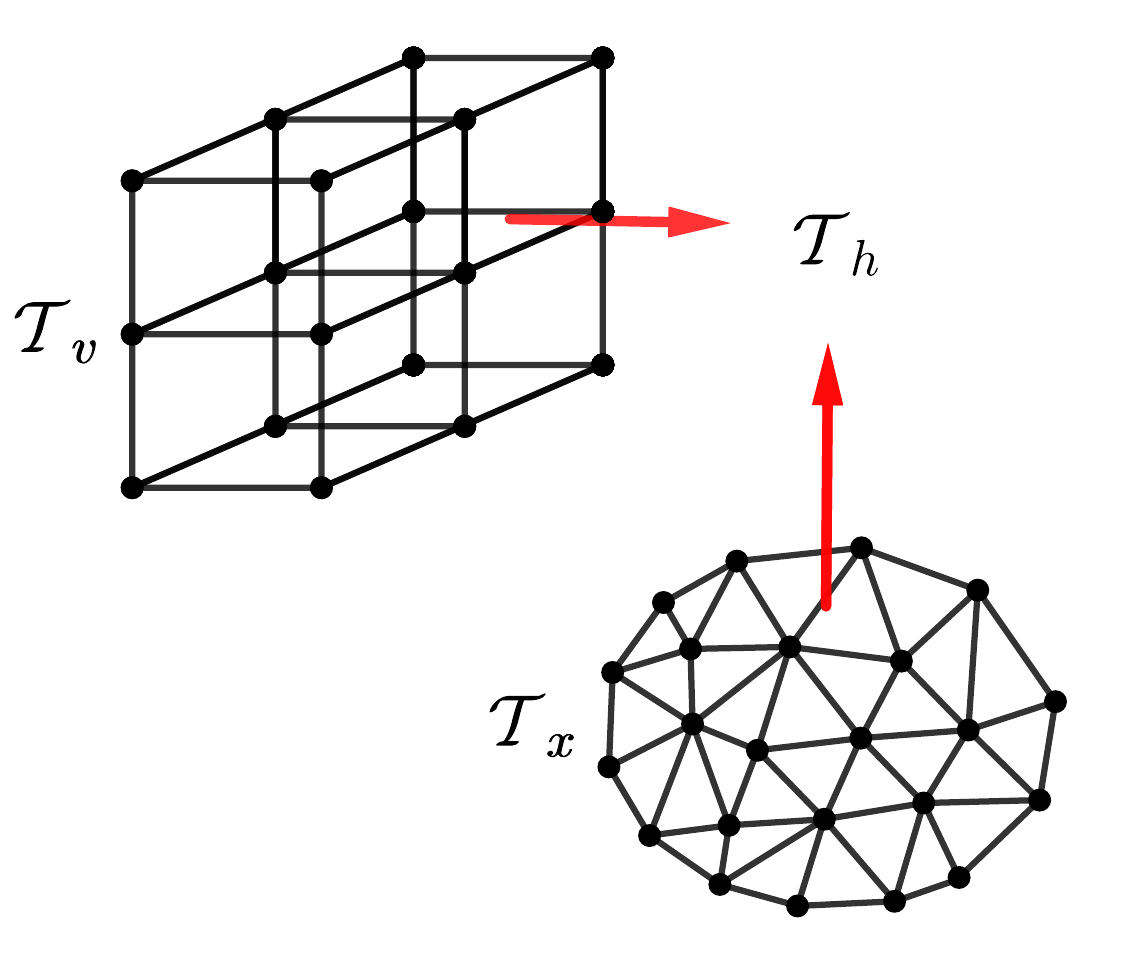}
    \caption{${\rm 2D3V},\polP_k\otimes\polQ_k$}
  \end{subfigure}\hfill
  \caption{Illustration of the tensor product meshes.}
  \label{fig:mesh}
\end{figure}

\section{Finite element framework for Vlasov--Maxwell equations}\label{sec:FEVM}
In this section, we introduce a framework for Vlasov–Maxwell equations using finite element methods. We use the finite element exterior calculus in this framework to preserve the structural properties of Maxwell’s equations. We also introduce the computation of charge and current density, which couple Maxwell’s equations with the distribution function. 

\subsection{First-order artificial viscosity}
The finite element approximation for \eqref{eq:vw:advection:visc} is given by multiplying it with test functions and integrating by parts, i.e., find $f_h({\pmb{x}} , {\pmb{v}} , t)\in \calC^1([0,T];\polV)$ such that
\begin{equation}\label{eq:GFEM:visc}
  (\p _t f_h+\pmb{\bbetaa}_h\SCAL \nabla f_h,\psi_i)+(\pmb{\mathsf{A}}\GRAD f_h, \GRAD \psi_i)
  =0,\qquad\forall i\in\calI,
\end{equation}
where $\nabla:=(\nabla_{{\pmb{x}}},\nabla_{{\pmb{v}}})^\top$, $\pmb{\bbetaa}_h:=({\pmb{v}},\bE_h+{\pmb{v}}\times\bB_h)^\top$, and $\pmb{\mathsf{A}}:= \textrm{diag}(\mathsf{A}_{\pmb{x}}, \mathsf{A}_{\pmb{v}})$, with $\mathsf{A}_{\pmb{x}}$ and $\mathsf{A}_{\pmb{v}}$ being mesh-dependent artificial viscosity coefficient matrices, and $\pmb{E}_h$ and $\bB_h$ are the FE approximations of $\bE$ and $\bB$, respectively, defined in the section below.

\textcolor{black}{
\begin{proposition}
The vector field $\pmb{\bbetaa}_h$ is divergence-free.
\end{proposition}
\begin{proof}
It is trivial to note that $\nabla_\bx \cdot \pmb{v} + \nabla_\bv \cdot \bE_h\equiv0$. For the remaining term $\nabla_\bv \cdot (\pmb{v} \times \bB_h)$, this follows, since $\bv \times \bB_h  \perp \bv$.
\end{proof}}

In time, we use high-order explicit Strong-Stability-Preserving Runge-Kutta (SSP-RK) methods. For the details on SSP-RK methods, see \cite{gottlieb2005high} and the references therein. These methods are constructed as a convex combination of the forward Euler method, therefore it is sufficient to present our method only for the forward Euler method. 

We start by discretizing the time interval as $0=t_0 < t_1 < \ldots < t_{N_T}$, and denote the local time step as $\tau_n := t_{n+1} - t_n$. The FE solution $f_h^n \approx f_h(t_n)$ can be written as $f_h^n = \sum_{i\in\calI}f_i^n\psi_i$, and let us discretize \eqref{eq:GFEM:visc} in time using the forward Euler scheme, then the following fully discrete system is obtained:
\begin{equation}\label{eq:FO}
  \sum_{j\in\calI} m_{ij}\frac{f^{n+1}_{j} - f^n_{j}}{\tau_n} 
  + 
  \sum_{j\in\calI} c^n_{ij} f^n_{j} 
  +
  \sum_{j\in\calI} d_{ij}^{n} f^n_{j} = 0,\qquad\forall i\in\calI,
\end{equation} 
where 
$
m_{ij} := (\psi_j, \psi_i)
$
is the mass matrix,
$c^n_{ij}:= ( \pmb{\bbetaa}_h^n\SCAL\GRAD \psi_j, \psi_i)$ is the advection matrix, and 
$
d_{ij}^n := (\pmb{\mathsf{A}}^n \GRAD\psi_j, \GRAD \psi_i)
$ is the diffusion matrix for all $\psi_i,\psi_j\in\polV$, $\pmb{\bbetaa}_h^n=\pmb{\bbetaa}_h(t_n)$ and $\pmb{\mathsf{A}}^n=\pmb{\mathsf{A}}(t_n)$. The time step size in \eqref{eq:FO} is under the control of the following CFL condition:
\begin{equation}\label{eq:CFL}
  \tau_n = \lambda\frac{h_{\textrm{min}}}{k\Vert\pmb{\bbetaa}_h^n\Vert_{L^\infty(\Omega)}},
\end{equation}
where 
$h_{\textrm{min}}:= \min_{i=1,\ldots,d_{\pmb{x}}, j=1,\ldots,d_{\pmb{v}}}(\Delta x_i, \Delta v_j)$ is the smallest mesh size in the mesh $\calT$, and $\lambda$ is the CFL number.

The artificial matrix $\pmb{\mathsf{A}}^n$ can be constructed as a first-order function which is proportional to the mesh size such as in \cite{wen2025anisotropicnonlinearstabilizationfinite}: 
$\mathsf{A}^{\textrm{L},n}_{\pmb{x}}\in[\polV]^{d_{\pmb{x}}}:= \textrm{diag}(\e^{\textrm{L},n}_{x_1}, \ldots, \e^{\textrm{L},n}_{x_{d_{\pmb{x}}}})$ and $\mathsf{A}^{\textrm{L},n}_{{\pmb{v}}}\in[\polV]^{d_{\pmb{v}}}:= \textrm{diag}(\e^{\textrm{L},n}_{v_1}, \ldots, \e^{\textrm{L},n}_{v_{d_{\pmb{v}}}})$, where the nodal values of $\e^{\textrm{L},n}_{x_l}$ and $\e^{\textrm{L},n}_{v_l}$ are given by
\begin{equation}\label{eq:AV:L}
  \begin{aligned}
    \big(\e^{\textrm{L},n}_{x_l}\big)_i&= \frac12 \frac{\Delta x_l}{k} \|\beta_{x_l}\|_{L^\infty(S(i))},\\
    \big(\e^{\textrm{L},n}_{v_l}\big)_i&= \frac12 \frac{\Delta v_l}{k} \|\beta_{v_l}\|_{L^\infty(S(i))},
  \end{aligned}
\end{equation}
for every node $i=1,\ldots,N$ and $l = 1, \dots, d_\bz$, $\bz\in\{{\pmb{x}},{\pmb{v}}\}$. Then, we define $\pmb{\mathsf{A}}^{\textrm{L},n}:= \textrm{diag}(\mathsf{A}^{\textrm{L},n}_{{\pmb{x}} }, \mathsf{A}^{\textrm{L},n}_{{\pmb{v}}})$, and the following result shows that the method is $L^2$-stable with this choice of artificial viscosity.
\begin{proposition}\label{prop:SE}
There is a uniform constant $C>0$, independent of the mesh size, such that the scheme \eqref{eq:FO} is stable in the sense that
\[
\|f_h\|^2_{L^{\infty}( (0,T); L^2(\Omega) )} + (1 - \lambda\textcolor{black}{C}) \sum_{n=0}^{N_T} 2 \tau_n 
\|(\pmb{\mathsf{A}}^{{\rm L},n})^{\frac12} \GRAD f^n_h\|^2 \le \|f_0\|^2,
\]
where $\lambda$ is the CFL number.
\end{proposition}

\begin{proof}
Let us write \eqref{eq:FO} in an equivalent integral form
\begin{equation*}
 \Big(\frac{f^{n+1}_h-f^n_h}{\tau_n}, \, w \Big) +
 \big(\pmb{\bbetaa}_h^n \SCAL \GRAD f^n_h, \, w \big) +
 \big(\pmb{\mathsf{A}}^{\textrm{L},n}\GRAD f^n_h, \, \GRAD w \big) 
 =0, \quad n = 0, 1, \ldots, N_T.
\end{equation*}
By inserting $w=f^n_h$ into the last equality and taking into account the identity
$
(a-b)b = \frac12(a^2 - b^2 - (a-b)^2),
$
as well as 
$
\big(\pmb{\bbetaa}_h^n \SCAL \GRAD f^n_h, \, f^n_h \big) = 0,
$
due to the divergence free property of $\pmb{\bbetaa}^n_h$ and the periodic boundary condition,
we obtain:
\begin{equation}\label{eq:FO:l2}
\frac{1}{2\tau_n} \big( \|f^{n+1}_h\|^2 - \|f^n_h\|^2 \big) 
+ \|(\pmb{\mathsf{A}}^{\textrm{L},n})^{\frac12} \GRAD f^n_h\|^2
=
\frac{\tau_n}{2} \Big\|\frac{f^{n+1}_h - f^n_h}{\tau_n} \Big\|^2.
\end{equation}
\textcolor{black}{
By noting the inverse estimate
$
\|\GRAD w \| \le C_1 k h_{\min}^{-1} \|w\|, 
$
for any $w\in \polV$, where $C_1>0$ is a constant, we have the following estimate for the discrete time-derivative 
 \begin{equation}\label{eq:Dtf}
\begin{aligned}
\Big\|\frac{f^{n+1}_h - f^n_h}{\tau_n} \Big\|
&\, = 
\sup_{w \in \polV \setminus \{0\}} \frac{1}{\|w\|} \Big| \Big( \frac{f^{n+1}_h - f^n_h}{\tau_n},\,  w\Big) \Big| \\
&\, \le 
\sup_{w \in \polV \setminus \{0\}} \frac{1}{\|w\|} 
\left(
\big| ( \pmb{\bbetaa}_h^n \SCAL \GRAD f^n_h,\, w ) \big|
+
\big|(
\pmb{\mathsf{A}}^{\textrm{L},n} \GRAD f^n_h, \, \GRAD w ) 
\big|
\right)
\\
&\, \le 
\sup_{w \in \polV \setminus \{0\}} \frac{1}{\|w\|} 
(
\Vert\pmb{\bbetaa}_h^n \SCAL \GRAD f^n_h\Vert \Vert w\Vert
+
\| \pmb{\mathsf{A}}^{\textrm{L},n} \GRAD f^n_h \| \, \|\GRAD w\|
)
\\
&\, \le 
\sup_{w \in \polV \setminus \{0\}} \frac{1}{\|w\|} 
\left(
\Vert\pmb{\bbetaa}_h^n \SCAL \GRAD f^n_h\Vert \Vert w\Vert
+
\| \pmb{\mathsf{A}}^{\textrm{L},n} \GRAD f^n_h \| \, C_1 k h_{\min}^{-1} \|w\|
\right)
\\
&\, = 
\|\pmb{\bbetaa}_h^n \SCAL \GRAD f^n_h\| + \textcolor{black}{}
C_1 k h_{\min}^{-1} \| \pmb{\mathsf{A}}^{\textrm{L},n} \GRAD f^n_h \| \, 
\\
&\, \le 
\|\pmb{\bbetaa}_h^n \SCAL \GRAD f^n_h\| + \textcolor{black}{}
C_2 k^\frac12 h_{\min}^{-\frac12} \|\pmb{\bbetaa}_h^n \|_{L^\infty(\Omega)}^{\frac12} \| (\pmb{\mathsf{A}}^{\textrm{L},n})^{\frac12} \GRAD f^n_h \|.
\end{aligned} 
\end{equation}
}
The flux term in the last inequality can be estimated as:
\[
\|\pmb{\bbetaa}_h^n \SCAL \GRAD f^n_h\| 
\le
\|\pmb{\bbetaa}_h^n\|_{L^\infty(\Omega)}
\|\GRAD f^n_h\|
\le
\|(\pmb{\mathsf{A}}^{\textrm{L},n})^{-1} \|_{L^\infty(\Omega)}^\frac12
\|\pmb{\bbetaa}_h^n\|_{L^\infty(\Omega)}
\|(\pmb{\mathsf{A}}^{\textrm{L},n})^{\frac12} \GRAD f^n_h\|.
\]
Each components $z\in\{x , v \}$ of the viscosity matrix can be estimated as
\[
(\e^{\textrm{L},n}_{z_l})^{-1} = 
\Big( \sum_{j\in\calI} \big(\e^{\textrm{L},n}_{z_l}\big)_j \psi_j \Big)^{-1}
\le
2 k\, h_{\min}^{-1} \Big( \min_{j\in\calI} \|\pmb{\bbetaa}_h^n\|_{L^\infty(S(j))} \Big)^{-1},
\]
for every $l=1,\ldots,d_z$, where we also used the partition of unity property of $\psi_j$. Thus,
\[
\|\pmb{\bbetaa}_h^n \SCAL \GRAD f^n_h\| 
\le
C_3 k^\frac12\, h_{\min}^{-\frac12} \|\pmb{\bbetaa}_h^n\|_{L^\infty(\Omega)}^\frac12
\|(\textcolor{black}{\pmb{\mathsf{A}}^{\textrm{L},n}})^{\frac12} \GRAD f^n_h\|,
\]
where
\textcolor{black}{
$
C_3 := 
\Big(
2 \|\pmb{\bbetaa}_h^n\|_{L^\infty(\Omega)} 
\big(
\min_{j\in\calI} \|\pmb{\bbetaa}_h^n\|_{L^\infty(S(j))}
\big)^{-1}
\Big)^{\frac12}
$
}. 
Now, we multiply \eqref{eq:FO:l2} by $2\tau_n$ and collect all terms in \eqref{eq:Dtf} and obtain
\begin{equation}\label{eq:case1}
\|f^{n+1}_h\|^2 
+ 2\tau_n
\Big(
1 - C \tau_n k h_{\min}^{-1} \|\pmb{\bbetaa}_h^n\|_{L^\infty(\Omega)}
\Big)
\|(\pmb{\mathsf{A}}^{\textrm{L},n})^{\frac12} \GRAD f^n_h\|^2
\le
\|f^n_h\|^2,
\end{equation}
where $C:=\textcolor{black}{\frac{1}{2}(C_2+C_3)^2}$ is independent of the mesh size, then sum \eqref{eq:case1} over $n=0,1,\ldots, N_T$, and get the desired estimate.
\hfill $\square$
\end{proof}

\subsection{Nonlinear residual-based viscosity}\label{sec:RV}
We have proved that the scheme \eqref{eq:GFEM:visc} with the artificial viscosity coefficients defined in \eqref{eq:AV:L} is stable in the previous section. However, this choice of viscosity coefficients results in only a first-order convergent scheme; as introduced in \cite{wen2025anisotropicnonlinearstabilizationfinite}, it yields upwind schemes. In this section, we present a nonlinear construction of the viscosity coefficients that take into account the finite element residual of the Vlasov--Maxwell equations, following the approach in \cite{MR3011445, Nazarov_Hoffman_2013, MR4754161} and other references cited therein, where similar methods were applied to the compressible Euler and magnetohydrodynamic equations.

In many previous works \textcolor{black}{(see, for example, \cite{wen2025anisotropicnonlinearstabilizationfinite,MR3998292, MR4754161,MR4664977, MR4456197,MR4734518})} we usually seek the residual of equation \eqref{eq:vw:advection} through the following $L^2$-projection: find $ R^n(f_h):=\sum_{j\in\calI} R^n_j\psi_j\in \polV$ such that 
\begin{equation}\label{eq:residual}
(R^n,\psi_i)=\big(\big| D_{\tau} f^n_h+\pmb{\bbetaa}_h^n\SCAL\nabla f_h^n\big|, \psi_i\big),\qquad\forall i\in\calI,
\end{equation}
where $D_{\tau} f^n_h$ is a reasonable approximation of the time-derivative of $f_h$ at time level $t^n$. \textcolor{black}{Additional operations may be applied to \eqref{eq:residual}, such as lumping the mass matrix or adding smoothing terms.} Once the finite element residual is approximated, we construct the following high-order artificial viscosity matrices: $\pmb{\mathsf{A}}^{\textrm{H},n}:= \textrm{diag}(\mathsf{A}^{\textrm{H},n}_{{\pmb{x}}}, \mathsf{A}^{\textrm{H},n}_{{\pmb{v}}})$, $\mathsf{A}^{\textrm{H}, n}_{{\pmb{x}}}\in[\polV]^{d_{\pmb{x}}}:= \textrm{diag}(\e^{\textrm{H},n}_{x_1}, \ldots, \e^{\textrm{H},n}_{x_{d_{\pmb{x}}}})$, and $\mathsf{A}^{\textrm{H},n}_{{\pmb{v}}}\in[\polV]^{d_{\pmb{v}}}:= \textrm{diag}(\e^{\textrm{H},n}_{v_1}, \ldots, \e^{\textrm{H},n}_{v_{d_{\pmb{v}}}})$, where the nodal values of $\e^{\textrm{H},n}_{x_l}$ and $\e^{\textrm{H},n}_{v_l}$ are given by
\begin{equation}\label{eq:AV:L=H}
\begin{aligned}
\big(\e^{\textrm{H},n}_{x_l}\big)_i 
&= 
\min\Big(\big(\e^{\textrm{L},n}_{x_l}\big)_i,\Big(\frac{\Delta x_l}{k}\Big)^{2} \frac{|R_i^n|}{\Lambda_\Omega(f^n_h)}\Big), \\ 
\big(\e^{\textrm{H},n}_{v_l}\big)_i 
&= 
\min\Big(\big(\e^{\textrm{L},n}_{v_l}\big)_i,\Big(\frac{\Delta v_l}{k}\Big)^{2} \frac{|R_i^n|}{\Lambda_\Omega(f^n_h)}\Big),
\end{aligned}
\end{equation}
for every node $i=1,\ldots,N$ and $l = 1, \dots, d_\bz$, $\bz\in\{{\pmb{x}},{\pmb{v}}\}$. Here, $\Lambda_\Omega(f^n_h)$ is the normalization constant at $t = t_n$, which we define below. This constant is used to remove the scale of $f_h$ from the residual. Various normalization functions have been presented in the literature where entropy and residual-based stabilization were used. In this work, we define
\begin{equation}
     \Lambda_Q(w):= \big \| w - \overline{w} \big \|_{L^\infty(Q)},
     \notag
\end{equation}
where $\overline{w} = \frac{1}{|Q|} \int_{Q} w$ is the average of the solution in the domain $Q$. In practice, to avoid division by zero, one can normalize the residual 
$\frac{R^n_i}{\Lambda_\Omega(f^n_h)} \approx R^n_i \frac{\Lambda_\Omega(f^n_h)}{(\Lambda_\Omega(f^n_h))^2  + 10^{-14} \|f^n_h\|^2_{L^{\infty}(\Omega)}}$. 

\textcolor{black}{
\begin{remark}
We use the first-order viscosity as an upper bound of viscosity coefficients to make sure our scheme is not overly diffusive and want to emphasize that we cannot establish a similar stability estimate as in Proposition \ref{prop:SE} for the residual viscosity. The main reason is the lack of a lower bound for the PDE residual. This implies that the residual can theoretically be zero, in which case the scheme essentially reduces to the standard Galerkin finite element method. However, in practice, if the PDE residual is zero or close to zero, it typically indicates that the solution is extremely smooth and therefore does not require additional stabilization.
In order to achieve theoretical completeness, one possible approach is to introduce a higher-order stabilization term that is activated when the residual vanishes. An example of such a mechanism is the edge stabilization method described in \cite{MR3062586}. However, this is beyond the scope of the present manuscript.
\end{remark}}

\subsection{A novel residual-based viscosity}\label{sec:newRV}
\textcolor{black}{
The method introduced in Section \ref{sec:RV} for computing the residual is robust for low-dimentional problems; while it is problematic when applied to high-dimensional cases. First, the backward differentiation formulas (BDFs) are applied for computing $D_\tau f_h^n$, and in our previous works, the experiments show that usually BDF2 is sufficient enough, which is given by ${\rm BDF}2(f_h)^n = (3f_h^n-4f_h^{n-1}+f_h^{n-2})/(2\tau_n)$, when the time-step is uniform, for the fourth-order accurate scheme.
The BDFs require storing the solutions at previous time steps, e.g., $f_h^{n-2}$, $f_h^{n-1}$ and $f_h^{n}$ are needed when using BDF2; the size of $f_h$ increases dramatically with the dimension, resulting in significantly higher storage requirement. 
In addition, projection step \eqref{eq:residual} requires solving a linear system, which can equally expensive as solving the Euler step of \eqref{eq:FO}.}

Due to this limitation, we introduce a novel approach for constructing the residual-based artificial viscosity suitable for high-dimensional problems. Let us start by taking the integral of \eqref{eq:vw:advection:visc} in $\Omega_{\pmb{x}}$ and $\Omega_{\pmb{v}}$ respectively, and obtain
\begin{equation}\label{eq:uv_ux}
  \begin{aligned}
    \partial_t u_v+\nabla_{\pmb{v}}\cdot\mathbf{F}_{\pmb{v}}(u_v)-\nabla_{\pmb{v}}\cdot(\mathsf{A}_{\pmb{v}}\nabla_{\pmb{v}} u_v)=0,\\
    \partial_t u_x+\nabla_{\pmb{x}}\cdot\mathbf{F}_{\pmb{x}}(u_x)-\nabla_{\pmb{x}}\cdot(\mathsf{A}_{\pmb{x}}\nabla_{\pmb{x}} u_x)=0,
  \end{aligned}
\end{equation}
where
\begin{equation}\label{eq:flux}
  \begin{aligned}
    u_v({\pmb{v}},t)=\int_{\Omega_{\pmb{x}}}f\ {\rm d}{\pmb{x}},\qquad
    \mathbf{F}_{\pmb{v}}(u_v)=\int_{\Omega_{\pmb{x}}}\bbetaa_{\pmb{v}} f\ {\rm d}{\pmb{v}},\\
    u_x({\pmb{x}},t)=\int_{\Omega_{\pmb{v}}}f\ {\rm d}{\pmb{v}},\qquad
    \mathbf{F}_{\pmb{x}}(u_x)=\int_{\Omega_{\pmb{v}}}\bbetaa_{\pmb{x}} f\ {\rm d}{\pmb{x}}.
  \end{aligned}
\end{equation}
In addition, we require the viscosity matrices $\mathsf{A}_{\pmb{x}}:=\mathsf{A}_{\pmb{x}}(u_x)$ and $\mathsf{A}_{\pmb{v}}:=\mathsf{A}_{\pmb{v}}(u_v)$ be fully anisotropic. For this purpose, we propose to regularize \eqref{eq:uv_ux} component-wise, where the viscosity matrices are constructed using the following corresponding residuals: find $ R_{x}^n(u_{h,x})\in \polV_{{\pmb{x}}}$ and $ R_{v}^n(u_{h,v})\in \polV_{{\pmb{v}}}$ such that 
\begin{equation}\label{eq:residual:density}
  \begin{aligned}
    (R^n_{x},\phi_i)&=\big(\big| D_{\tau} u^n_{h,x}+\nabla_{\pmb{x}}\cdot\mathbf{F}_{h,{\pmb{x}}}^n\big|, \phi_i\big),\qquad\forall i\in\calI_{\pmb{x}},\\
    (R^n_{v},\varphi_j)&=\big(\big| D_{\tau} u^n_{h,v}+\nabla_{\pmb{v}}\cdot\mathbf{F}_{h,{\pmb{v}}}^n\big|, \varphi_j\big),\qquad\forall j\in\calI_{\pmb{v}},
  \end{aligned}
\end{equation}
where $u^n_{h,x}$ and $u^n_{h,v}$ are finite element approximation of $u^n_{x}(t^n,{\pmb{x}})$ and $u_{v}(t^n,{\pmb{v}})$ respectively, $D_{\tau} u^n_{h,x}$ and $D_{\tau} u^n_{h,v}$ are approximations of $\partial_t u_{h,x}$ and $\partial_t u_{h,v}$ at time level $t^n$, which are computed using BDF2; $u^n_{h,x}$, $u^n_{h,v}$, $\mathbf{F}^n_{h,{\pmb{x}}}$ and $\mathbf{F}^n_{h,{\pmb{v}}}$ are derived from $f_h^n$ and $\pmb{\bbetaa}_h^n$ using the formulation \eqref{eq:flux}. \textcolor{black}{Note that the projection problems in \eqref{eq:residual:density} are low-dimensional and can be computed very efficiently.}

In order to make the viscosity anisotropic, the first-order viscosity is reduced to depend only on the spatial or velocity space by taking the average values over the other direction, i.e.,
\begin{equation}
  \begin{aligned}
  \big(\nu^{\textrm{L},n}_{x_l}\big)_i 
  &= \frac{1}{N_\bv}\sum_{j\in\calI_{\pmb{v}}}\big(\e^{\textrm{L},n}_{x_l}\big)_{i,j},\qquad \forall i\in\calI_{\pmb{x}},\\ 
  \big(\nu^{\textrm{L},n}_{v_l}\big)_i 
  &= \frac{1}{N_\bx}\sum_{j\in\calI_{\pmb{x}}}\big(\e^{\textrm{L},n}_{x_l}\big)_{j,i},\qquad \forall i\in\calI_{\pmb{v}}.
  \end{aligned}\notag
\end{equation}
In addition, since the residual viscosity is derived from the residuals of \eqref{eq:uv_ux}, the normalization function should also be defined in terms of $u^n_{h,x}$ and $u^n_{h,v}$. Following the approach in \cite{Dao2022a,WEN2025114079}, we introduce a parameter that depends on local discontinuities and define the following normalization function:
\begin{equation}
  \begin{aligned}
    \Lambda_{\bz_i}=\left(1-0.5\frac{\max_{j\in S_z(i)}u^n_{z,j}-\min_{j\in S_z(i)}u^n_{z,j}}{\max_{j\in\calI_\bz}u^n_{z,j}-\min_{j\in\calI_\bz}u^n_{z,j}}\right)\big \|u^n_{h,z} - \overline{u^n_{h,z}} \big \|_{L^\infty(\Omega_\bz)},\quad\forall i\in\calI_\bz,\\
  \end{aligned}\notag
\end{equation}
for $\bz \in \{ \bx, \bv\}$, and then we compute the high-order viscosity whose nodal values are given by
\begin{equation}\label{eq:AV:L=H2}
  \begin{aligned}
  \big(\nu^{\textrm{H},n}_{x_l}\big)_i 
  &= 
  \min\Big( \big(\nu^{\textrm{L},n}_{x_l}\big)_i ,\, \Big(\frac{\Delta x_l}{k}\Big)^{2} \frac{\vert R^n_{x,i}\vert}{\Lambda_{\bx_i}}\frac{d_\bx}{d_\bx+d_\bv}\Big),\qquad \forall i\in\calI_{\pmb{x}}, l = 1,\ldots,d_\bx,\\ 
  \big(\nu^{\textrm{H},n}_{v_l}\big)_i 
  &= 
  \min\Big( \big(\nu^{\textrm{L},n}_{v_l}\big)_i ,\, \Big(\frac{\Delta v_l}{k}\Big)^{2} \frac{\vert R^n_{v,i}\vert}{\Lambda_{\bv_i}}\frac{d_\bv}{d_\bx+d_\bv}\Big),\qquad \forall i\in\calI_{\pmb{v}}, l = 1,\ldots,d_\bv. 
  \end{aligned}
\end{equation}
Finally, we define the new viscosity matrices as follows: $\pmb{\mathsf{A}}^{\textrm{H},n}:= \textrm{diag}(\mathsf{A}^{\textrm{H},n}_{\pmb{x}},\mathsf{A}^{\textrm{H},n}_{\pmb{v}})$, 
$\mathsf{A}^{\textrm{H},n}_{\pmb{x}}\in[\polV_{\pmb{x}}]^{d_{\pmb{x}}}:= \textrm{diag}(\nu^{\textrm{H},n}_{x_1}, \ldots, \nu^{\textrm{H},n}_{x_{d_{\pmb{x}}}})$, and $\mathsf{A}^{\textrm{H},n}_{{\pmb{v}}}\in[\polV_{\pmb{v}}]^{d_{\pmb{v}}}:= \textrm{diag}(\nu^{\textrm{H},n}_{v_1}, \ldots, \nu^{\textrm{H},n}_{v_{d_{\pmb{v}}}})$.
\textcolor{black}{
\begin{remark}
  The new formulation \eqref{eq:residual:density} for computing the residual is more efficient than the conventional approach \eqref{eq:residual}, as the computational cost of \eqref{eq:residual:density} scales with $N_\bx + N_\bv$, whereas the cost of \eqref{eq:residual} scales with $N_\bx N_\bv$. This leads to a significantly lower computational expense, especially when both $N_\bx$ and $N_\bv$ are large. The robustness of the new definition of residual viscosity is verified through numerical experiments. For example, Figure~\ref{fig:LD} shows that in the presence of a strong shock, the viscosity is sufficient to stabilize the FE solutions.
\end{remark}}
\begin{remark}
  The artificial viscosity introduces diffusion into the system. To compute the viscosity coefficients, we use the PDE residuals such that the artificial viscosity vanishes in smooth regions. In this work, we construct the viscosity in each direction based on the residuals of \eqref{eq:uv_ux}. However, since diffusion accumulates across all directions, we apply the scaling factors $\frac{d_\bx}{d_\bx + d_\bv}$ and  $\frac{d_\bv}{d_\bx + d_\bv}$  to balance the contribution. This choice ensures consistency with the conventional residual viscosity method. For instance, when $d_\bv = 0$, we have $\frac{d_\bx}{d_\bx + d_\bv} = 1$ and $\frac{d_\bv}{d_\bx + d_\bv} = 0$.
\end{remark}
\begin{proposition}
The forward Euler scheme \eqref{eq:FO} is conservative if we use the artificial viscosity defined in Section \ref{sec:newRV}.
\end{proposition}
\begin{proof}
  We refer to the proof in \cite[Prop. 3]{WEN2025114079}, and in this work, the new artificial viscosity defined in Section \ref{sec:newRV} preserves the property $\sum_{j=1}^N d_{ij}^n = \sum_{i=1}^N d_{ij}^n = 0$, so the same argument applies here as well. 
  \hfill $\square$
\end{proof}

\subsection{Tensor product FE approximations for Vlasov--Maxwell equations}
We now introduce the approach for implementing the finite element methods introduced above using the Kronecker product of matrices. The function $f_h$ can be denoted by $f_h(\pmb{x},\pmb{v},t)=\sum_{j\in\calI} f_j(t)\psi_j(\pmb{x},\pmb{v})$, and the equation \eqref{eq:GFEM:visc} can be reformulated as follows
\begin{equation}\label{eq:GFEM:ts}
  \begin{aligned}
    \bigg(\partial_t\sum_{j\in\calI} f_j\psi_j,\psi_i\bigg)+\bigg(\pmb{v}\cdot\nabla_{\pmb{x}}\sum_{j\in\calI}f_j\psi_j,\psi_i\bigg)+\bigg(\pmb{E}_h\cdot\nabla_{\pmb{v}}\sum_{j\in\calI}f_j\psi_j,\psi_i\bigg)\\
    +\bigg(\pmb{B}_h\cdot\Big(\nabla_{\pmb{v}}\sum_{j\in\calI}f_j\psi_j\times\pmb{v}\Big),\psi_i\bigg)+\bigg(\mathsf{A}_{{\pmb{x}}}\Big(\nabla_{\pmb{x}}\sum_{j\in\calI}f_j\psi_j\Big),\nabla_{\pmb{x}}\psi_i\bigg)\\+\bigg(\mathsf{A}_{{\pmb{v}}}\Big(\nabla_{\pmb{v}}\sum_{j\in\calI}f_j\psi_j\Big),\nabla_{\pmb{v}}\psi_i\bigg)=0,\qquad\forall i\in\calI.
  \end{aligned}
\end{equation}
We sort the basis functions $\{\psi_i\}_{i\in\calI}$ using such index: 
\begin{equation}
  \{\psi_1,\cdots\psi_{N_{\pmb{v}}},\cdots,\psi_{(N_{\pmb{x}}-1)N_{\pmb{v}}+1},\cdots,\psi_{N_{\pmb{x}} N_{\pmb{v}}}\},\notag
\end{equation}
which is equivalent to 
\begin{equation}
  \{\phi_1\varphi_1,\ldots,\phi_1\varphi_{N_{\pmb{v}}},\ldots,\phi_{N_{\pmb{x}}}\varphi_1,\ldots,\phi_{N_{\pmb{x}}}\varphi_{N_{\pmb{v}}}\}.\notag
\end{equation}
Let the vector $\polf=(f_1,\cdots,f_N)^{\mathsf{T}}$ be the collection of all the degrees of freedom of $f_h$, ordered as previously introduced. 

The equation \eqref{eq:GFEM:ts} yields the following system
\begin{equation}\label{eq:system}
  \begin{aligned}
    (\polM^{{\pmb{x}}}\otimes \polM^{{\pmb{v}}})\dot{\polf} + 
    \sum_{l=1}^3 
    \Big( 
      \polA^{\pmb{x},l}\otimes\polC^{\pmb{v},l}
      & + 
      \polC^{\pmb{x},l}(\bE)\otimes \polA^{\pmb{v},l} + 
      \polC^{\pmb{x},l}(\bB)\otimes \polG^{\pmb{v},l}
    \Big)  
      \polf\\
    & +
     (\polD^{{\pmb{x}}}\otimes \polM^{{\pmb{v}}})\polf+(\polM^{{\pmb{x}}}\otimes \polD^{{\pmb{v}}})\polf= 0,
  \end{aligned}
\end{equation}
where the operators in above equation are defined as follows
  \begin{align*}
    \polA_{ij}^{{\pmb{x}},l} &= \int_{\Omega_{\pmb{x}}}\textcolor{black}{(}\partial_{x_l}\phi_j\textcolor{black}{)}\phi_i\ {\rm d}\pmb{x}, \quad
    \polC_{ij}^{{\pmb{x}},l}(\bw) = \int_{\Omega_{\pmb{x}}} w_l\phi_j\phi_i\ {\rm d}\pmb{x}, 
    \quad l = 1,2,3,\\
    \polM_{ij}^{{\pmb{x}}} &= \int_{\Omega_{\pmb{x}}}\phi_j\phi_i\ {\rm d}\pmb{x}, \quad
    \polD_{ij}^{{\pmb{x}}}= \int_{\Omega_{\pmb{x}}}\mathsf{A}_{{\pmb{x}}}\textcolor{black}{(}\nabla_{\pmb{x}}\phi_j\textcolor{black}{)}\SCAL\textcolor{black}{(}\nabla_{\pmb{x}}\phi_i\textcolor{black}{)}\ {\rm d}\pmb{x}, \quad
    \forall i,j\in\calI_{\pmb{x}},
  \end{align*}
and
  \begin{align*}
    \polA_{ij}^{{\pmb{v}},l} &= \int_{\Omega_{\pmb{v}}}\textcolor{black}{(}\partial_{v_l}\varphi_j\textcolor{black}{)}\varphi_i\ {\rm d}\pmb{v}, \quad 
    \polC_{ij}^{{\pmb{v}},l} = \int_{\Omega_{\pmb{v}}} v_l\varphi_j\varphi_i\ {\rm d}\pmb{v}, \\
    \polG_{ij}^{\pmb{v},l} &= \int_{\Omega_{\pmb{v}}} \big(\nabla_{\pmb{v}} \CROSS \pmb{v} \big)_l \varphi_j \varphi_i {\rm d} \pmb{v} , 
    \quad l = 1,2,3, \\
    \polM_{ij}^{{\pmb{v}}} &=\int_{\Omega_{\pmb{v}}}\varphi_j\varphi_i\ {\rm d}\pmb{v}, \quad 
    \polD_{ij}^{{\pmb{v}}}= \int_{\Omega_{\pmb{v}}}\mathsf{A}_{{\pmb{v}}}\textcolor{black}{(}\nabla_{\pmb{v}}\varphi_j\textcolor{black}{)}\SCAL\textcolor{black}{(}\nabla_{\pmb{v}}\varphi_i\textcolor{black}{)}\ {\rm d}\pmb{v}, \quad \forall i,j\in\calI_{\pmb{v}}.
  \end{align*}
\begin{remark}
  In reduced Vlasov--Maxwell systems, equations \eqref{eq:system} will simplify to the equations introduced in section \ref{sec:reducedVM}.
\end{remark}

\subsection{Finite element approximation of the Maxwell’s equations}
Once the approximation $f_h^{n} \in \polV$ is obtained, according to \eqref{eq:denst}, we can compute the charge density and current density as follows
\begin{equation}
  \rho_h^n = \int_{\Omega_{\pmb{v}}}f_h^n\ {\rm d}\pmb{v},\qquad \pmb{J}_h^n = \int_{\Omega_{\pmb{v}}}\pmb{v}f_h^n\ {\rm d}\pmb{v}\notag.
\end{equation}
Since $f_h^n\in\polV:=\polV_{\pmb{x}}\otimes\polV_{\pmb{v}}$, it naturally holds true that $\rho_h^n\in\polV_{\pmb{x}}$ and $\pmb{J}_h^n\in[\polV_{\pmb{x}}]^{d_{\pmb{v}}}$. Now recall the modified current density \eqref{eq:J_h:visc}, we replace $\bJ^n_h$ with the following object
\begin{equation}\label{eq:current_mod}
  \widetilde{\bJ_h^n} =\bJ_h^n - \int_{\Omega_{\pmb{v}}}\mathsf{A}^n_{{\pmb{x}}}\GRAD_{{\pmb{x}}}f_h^n\ {\rm d}{\pmb{v}}= \bJ_h^n - \mathsf{A}^n_{{\pmb{x}}}\GRAD_{{\pmb{x}}}\rho_h^n.
\end{equation}

\label{sec:fem:maxwell}
Let us define an anisotropic polynomial space $\polQ_{\alpha_1, \alpha_2, ..., \alpha_{d_{\pmb{x}}}}$ composed of $d_{\pmb{x}}$-variate polynomials whose degree with respect to $x_i$ is at most $\alpha_i$. Then, we define the following spaces adopted from \cite[Sec.~3.9]{Monk_2003}:
\[
\pmb{\polE}:=
\{
\bw\in \bH(\textrm{curl};\Omega_{{\pmb{x}} }): [\GRAD_{\widehat{{\pmb{x}} }}\bT_K(\widehat{{\pmb{x}} }))]^\top \bw(\bT_K({\widehat{{\pmb{x}} }})) \in \calN_{k-1}(\widehat{K}), \forall K\in \calT_{{\pmb{x}} }
\},
\]
and
\[
\begin{aligned}
\pmb{\polB}:=
\{
\bw & \in \bH(\textrm{div};\Omega_{{\pmb{x}} }): \\
&
\det(\GRAD_{\widehat{{\pmb{x}} }}\bT_K(\widehat{{\pmb{x}} }))
[\GRAD_{\widehat{{\pmb{x}} }}\bT_K(\widehat{{\pmb{x}} }))]^{-1} 
\bw(\bT_K({\widehat{{\pmb{x}} }})) \in \calR\calT_{k-1}(\widehat{K}), 
 \forall K\in \calT_{{\pmb{x}} }
\},
\end{aligned}
\]
where
\[
\begin{aligned}
\calN_{k}(\widehat{K}) 
&:= 
\big[ \polQ_{k, k+1, k+1}(\widehat{K}), \polQ_{k+1, k, k+1}(\widehat{K}), \polQ_{k+1, k+1, k}(\widehat{K}) \big], \\
\calR\calT_{k}(\widehat{K}) 
&:= 
\big[ 
\polQ_{k+1, k, k}(\widehat{K}), 
\polQ_{k, k+1, k}(\widehat{K}), 
\polQ_{k, k, k+1}(\widehat{K}) \big] 
\end{aligned}
\]
are the well-known N\'ed\'elec and Raviart-Thomas elements of order $k$ for cartesian grids for the case of $d_{\pmb{x}}=3$.

With these spaces at hand, we obtain the finite element discretization of the Maxwell's equations \eqref{eq:AP:visc} as follows: for given $f^n_h \in \polV$, $\rho^n_h \in \polV_{\pmb{x}}$, $\pmb{J}^n_h \in [\polV_{\pmb{x}}]^{d_{\pmb{v}}}$, $\bE^n_h \in \pmb{\polE}$ and $\bB^n_h \in \pmb{\polB}$, we are looking for $\bE^{n+1}_h \in \pmb{\polE}$ and $\bB^{n+1}_h\in\pmb{\polB}$ from the following forward Euler scheme: 

\begin{equation}\label{eq:fem:AP11}
\begin{aligned}
  \Big(
	\frac{\bE^{n+1}_h - \bE^{n}_h}{\tau_n},\, \pmb{\eta} 
  \Big)
  &=
  \big(
  c^2\bB^n_h,\nabla_{\pmb{x}}\times\pmb{\eta}\big)-\big(\epsilon_0^{-1}\big(\bJ^n_h-\mathsf{A}^n_{{\pmb{x}}}\GRAD_{{\pmb{x}}}\rho^n_h\big),\ \pmb{\eta}
  \big), \quad \forall \pmb{\eta} \in \pmb{\polE},
  \\
	\Big(
	\frac{\bB^{n+1}_h - \bB^{n}_h}{\tau_n},\, \pmb{\xi} 
  \Big)
  &=
  \big(
  -\nabla_{{\pmb{x}} }\CROSS\bE^n_h,\, \pmb{\xi}
  \big), \quad \forall \pmb{\xi} \in \pmb{\polB}.
\end{aligned}
\end{equation}
Moreover, the forward Euler discretization of \eqref{eq:ctcharge:visc} gives us
\begin{equation}\label{eq:discreteContinuityEqn}
    \big(\GRAD_{{\pmb{x}}}\SCAL \big(\bJ^n_h-\mathsf{A}^n_{{\pmb{x}}}\GRAD_{{\pmb{x}}}\rho^n_h\big),\phi_i\big)=
  -\Big(
	\frac{\rho^{n+1}_h - \rho^{n}_h}{\tau_n},\, \phi_i
  \Big),\qquad\forall i\in\calI_\bx,
\end{equation}
and we introduce the computation of $\rho_h^n$ and $\bJ^n_h$ in section \ref{sec:density}.

\begin{proposition}[Discrete Gauss' law] Suppose the initial condition is chosen such that the Gauss' law \eqref{eq:GL} is satisfied and the discrete continuity equation \eqref{eq:discreteContinuityEqn} holds for all time steps. Then, the finite element solutions $\bE_h^{n+1}$ and $\bB_h^{n+1}$ of the scheme \eqref{eq:fem:AP11} satisfy the Gauss' law weakly, \ie
\begin{equation}\label{eq:weak:GL}
\begin{aligned}
\Big(
\frac{\bE^{n+1}_h - \bE^{n}_h}{\tau_n}, \, \GRAD_{{\pmb{x}}}\phi_i
\Big) &= -\Big(
	\epsilon_0^{-1}\frac{\rho^{n+1}_h - \rho^{n}_h}{\tau_n},\,\phi_i 
  \Big), \qquad \forall i\in \calI_\bx,\\
\GRAD_{{\pmb{x}}} \SCAL \bB_h^{n+1} &= 0, \mbox{ a.e.}. \\
\end{aligned}
\end{equation}
\end{proposition}

\begin{proof}
Since, for any function $\phi_i\in\polV_\bx$, we have $\nabla_{\pmb{x}}\phi_i\in \pmb{\polE}$. Thus, we test \eqref{eq:fem:AP11} with the test function $\pmb{\eta} = \GRAD_{{\pmb{x}}}\phi_i$ and obtain initially:
\[
  \Big(
	\frac{\bE^{n+1}_h - \bE^{n}_h}{\tau_n},\, \GRAD_{{\pmb{x}}}\phi_i
  \Big)
  =
  \big(
  c^2\bB^n_h,\nabla_{{\pmb{x}} }\CROSS(\nabla_{\pmb{x}}\phi_i)\big)-\big(\epsilon_0^{-1}\big(\bJ^n_h-\mathsf{A}^n_{{\pmb{x}}}\GRAD_{{\pmb{x}}}\rho^n_h\big),\ \GRAD_{{\pmb{x}}}\phi_i
  \big).  
\]
Integration by parts of the right hand side term, and taking into account the periodic boundary conditions, we obtain
\begin{equation*}\label{eq:prop:gl:E}
  \begin{aligned}
  \Big(
	\frac{\bE^{n+1}_h - \bE^{n}_h}{\tau_n},\, & \GRAD_{{\pmb{x}}}\phi_i
  \Big)
   =\\
   & \big(
  c^2\bB^n_h,\nabla_{{\pmb{x}} }\CROSS(\nabla_{\pmb{x}}\phi_i)\big)+\big(\epsilon_0^ {-1}\GRAD_{{\pmb{x}}}\SCAL\big(\bJ^n_h -\mathsf{A}^n_{{\pmb{x}}}\GRAD_{{\pmb{x}}}\rho^n_h\big),\phi_i
  \big).  
\end{aligned}
\end{equation*}
Note that, the first term of the right hand side is zero because $\nabla_{\pmb{x}}\times(\nabla_{\pmb{x}}\phi_i)=0$ almost everywhere, $\forall w\in\polV_{\pmb{x}}$, and for the second term we can use  \eqref{eq:discreteContinuityEqn} to find
\[
  \Big(
	\frac{\bE^{n+1}_h - \bE^{n}_h}{\tau_n},\, \GRAD_{{\pmb{x}}}\phi_i 
  \Big)
  =
-\Big(
	\epsilon_0^{-1}\frac{\rho^{n+1}_h - \rho^{n}_h}{\tau_n},\,\phi_i  
  \Big), \qquad \forall i \in \calI_\bx.
\]
It proves the first equality of \eqref{eq:weak:GL}, which is a discrete weak version of the Gauss' law \eqref{eq:GL_E_dt}.

To prove the second statement in \eqref{eq:weak:GL}, we observe that $\nabla_{{\pmb{x}}} \times \bE^n_h \in \pmb{\polB}$, which implies 
\[
	\frac{\bB^{n+1}_h - \bB^{n}_h}{\tau_n}
  +\nabla_{{\pmb{x}} }\CROSS\bE^n_h = 0, \quad \mbox{ a.e.}.
\]

Now, taking the divergence of this equality and noting $\nabla_{{\pmb{x}}} \cdot \bB^{n}_h = 0$ and $\nabla_{{\pmb{x}}} \cdot (\nabla_{{\pmb{x}}} \times \bE^n_h) = 0$ almost everywhere, we conclude that $\nabla_{{\pmb{x}}} \cdot \bB^{n+1}_h = 0$ almost everywhere.
\end{proof}

\subsection{Computation of the densities}\label{sec:density}
In the system of Vlasov--Maxwell equations, the charge density $\rho$ and current density $\pmb{J}$ are derived from the distribution function $f$. In addition, in the section \ref{sec:RV}, we have introduced the residual-based viscosity using the residuals of the conservation laws \eqref{eq:uv_ux}, where $u_{x/v}$ and $\mathbf{F}_{{\pmb{x}}/{\pmb{v}}}$ also come from $f$. In the case of Vlasov--Maxwell, $u_x = \rho$ and $\mathbf{F}_\bx=\pmb{J}$. In this section, we introduce the algorithms used to compute these quantities. 

The function $f_h^n$ can be expressed as follows 
\begin{equation}
  \begin{aligned}
    f_h^n&=\sum_{l\in\calI}\psi_lf_l=\sum_{i\in\calI_{\pmb{x}}}\sum_{j\in\calI_{\pmb{v}}}\psi_{(i-1)N_{\pmb{v}}+j}f_{(i-1)N_{\pmb{v}}+j}\\
         &=\sum_{i\in\calI_{\pmb{x}}}\sum_{j\in\calI_{\pmb{v}}}\phi_i\varphi_jf_{i,j},
  \end{aligned}\notag
\end{equation}
using this formulation, we can compute the charge density by simply computing the following expression

\begin{equation}
  \begin{aligned}
  \rho^n_h({\pmb{x}}) 
  &\,=
  \int_{\Omega_{{\pmb{v}}}} f^n_h({\pmb{x}}, {\pmb{v}}) \ud {\pmb{v}}=\int_{\Omega_{\pmb{v}}}\sum_{l\in\calI}\psi_i({\pmb{x}},{\pmb{v}})f_l^n\ {\rm d}{\pmb{v}} \\
  &\,=\int_{\Omega_{\pmb{v}}}\sum_{i\in\calI_{\pmb{x}}}\sum_{j\in\calI_{\pmb{v}}}\phi_i({\pmb{x}})\varphi_j({\pmb{v}})f_{i,j}^n\ {\rm d}{\pmb{v}}\\
  &\,=\sum_{i\in \calI_{\pmb{x}}}\left(\phi_i({\pmb{x}})\int_{\Omega_{{\pmb{v}}}}\sum_{j\in\calI_{\pmb{v}}}\varphi_j({\pmb{v}})f_{i,j}^n\ud {\pmb{v}}\right),
  \end{aligned}\notag
  \end{equation}
similarly, we compute the current density by
  \begin{equation}
  \begin{aligned}
  \bJ_h^n({\pmb{x}}) 
  &\,=
  \int_{\Omega_{{\pmb{v}}}} {\pmb{v}} f^n_h({\pmb{x}}, {\pmb{v}}) \ud {\pmb{v}} \\
  &\,=\sum_{i\in \calI_{\pmb{x}}}\left(\phi_i({\pmb{x}})\int_{\Omega_{{\pmb{v}}}}{\pmb{v}}\sum_{j\in\calI_{\pmb{v}}}\varphi_j({\pmb{v}})f_{i,j}^n\ud {\pmb{v}}\right).
  \end{aligned}\notag
  \end{equation}
Therefore, the nodal values of $\rho_h^n$ and $\pmb{J}_h^n$ are computed as follows
\begin{equation}\label{eq:int:v}
  \rho^n_i = \int_{\Omega_{{\pmb{v}}}}f^n_{i,\cdot}({\pmb{v}})\ud {\pmb{v}},\qquad\pmb{J}^n_i =\int_{\Omega_{{\pmb{v}}}}{\pmb{v}} f^n_{i,\cdot}({\pmb{v}})\ud {\pmb{v}},\qquad\forall i\in\calI_{\pmb{x}},
\end{equation}
where $f^n_{i,\cdot}\in\polV_{{\pmb{v}}}$ and $f^n_{i,\cdot} = \sum_{j\in\calI_{\pmb{v}}}\varphi_jf_{i,j}^n$. The discrete form of $u_v$ and $\mathbf{F}_{\pmb{v}}$ are given by:
\begin{equation}
    u_{h,v}^n= \int_{\Omega_{\pmb{x}}}f_h^n\ {\rm d}{\pmb{x}},\qquad \mathbf{F}_{h,{\pmb{v}}}^n= \int_{\Omega_{\pmb{v}}}\left(\bE_h^n f_h^n+({\pmb{v}}\times\bB_h^n) f_h^n\right)\ {\rm d}{\pmb{x}},\notag
\end{equation}
and the nodal values of $u^n_{h,v}$ and $\mathbf{F}^n_{h,\bv}$ can be obtained as follows
\begin{equation}\label{eq:int:x}
  u^n_{v,i} = \int_{\Omega_{{\pmb{x}}}}f^n_{\cdot,i}({\pmb{x}})\ud {\pmb{x}},\qquad\mathbf{F}^n_{\bv,i} =\int_{\Omega_{{\pmb{x}}}}(\pmb{E}_h^nf^n_{\cdot,i}({\pmb{x}})+(\pmb{v}\times\pmb{B}_h^n)f^n_{\cdot,i}({\pmb{x}}))\ud {\pmb{x}},\qquad\forall i\in\calI_{\pmb{v}},
\end{equation}
where $f^n_{\cdot,i}\in\polV_{{\pmb{x}}}$ and $f^n_{\cdot,i} = \sum_{j\in\calI_{\pmb{x}}}\phi_if_{j,i}^n$.
\begin{remark}
  We refer to the method introduced in \cite[Section~3.4]{WEN2025114079}, where an algorithm for computing $\rho_h$ is presented in the context of the Vlasov--Poisson system. In both algorithms, the density quantities are computed exactly; the main difference lies in the implementation details. In this work, the function spaces $\polV_{{\pmb{x}}}$ and $\polV_{{\pmb{v}}}$ are defined separetely, which makes the computations in \eqref{eq:int:v} and \eqref{eq:int:x} straightforward, and we implement our methods using FEniCSx \cite{barrata2023dolfinx}.
\end{remark}

\section{Numerical examples}\label{sec:exp}
In this section, we present the numerical results of our methods in various benchmark problems of Vlasov--Maxwell equations. We use the fourth-order, five-stage SSP-RK method from \cite{Kraaijevanger_1991} for all the experiments in this manuscript. The time step is controlled by the CFL condition \eqref{eq:CFL}; the CFL number we use is $\lambda=0.4$. The time-stepping algorithm of the methods can be described by Algorithm \ref{alg:FEVM}. We note that the artificial viscosity matrix is updated after each time step iteration. 

The finite element method is a grid-based scheme, and in our experiments, we manually cut the velocity space. The distribution functions are initialized with Gaussian perturbations in velocity space, with the velocity boundaries defined such that the value of $f$ is neglectable at the edges of the velocity space.

\begin{algorithm}[hbt!]
  \caption{Time-stepping procedure}\label{alg:FEVM}
  \SetKwInput{KwInput}{Input}
  \SetKwInput{KwOutput}{Output}
  \KwInput{Initial conditions: $f_0$, $\pmb{E}_0$, $\pmb{B}_0$; Final time: $T$}
  \KwOutput{$f_h(t = T)$, $\pmb{E}_h(t = T)$, $\pmb{B}_h(t = T)$}
  
  Initialize: $f_h^0 \leftarrow f_0$, $\pmb{E}_h^0 \leftarrow \pmb{E}_0$, $\pmb{B}_h^0 \leftarrow \pmb{B}_0$, $t \leftarrow 0$, $n \leftarrow 0$\;
  
  \While{$t < T$}{
    Compute $u^{n}_{h,x}$, $u^{n}_{h,v}$, $\mathbf{F}^{n}_{h,{\pmb{x}}}$, and $\mathbf{F}^{n}_{h,{\pmb{v}}}$\ using \eqref{eq:int:v} and \eqref{eq:int:x}\;
    Compute the residual $R_{x}^{n}$ and $R_{v}^{n}$ using \eqref{eq:residual:density}\;
    Compute the viscosity $\pmb{\mathsf{A}}^{n}$ using \eqref{eq:AV:L=H2}\;
    Determine the time step $\Delta t = \tau_n$ using \eqref{eq:CFL}\;
    Apply the Runge-Kutta scheme to update $f_h^{n+1}$, $\pmb{E}_h^{n+1}$, and $\pmb{B}_h^{n+1}$\;
    Update time: $t \leftarrow t + \Delta t$, $n \leftarrow n + 1$\;
  }
\end{algorithm}

\subsection{Vlasov--Maxwell in 1D2V and 2D2V phase-space}\label{sec:reducedVM}
Assuming all physical parameters are set to 1, we consider a single species and apply our methods to the Vlasov–Maxwell equations in the following reduced phase spaces.

Let $\Omega := \Omega_x\otimes\Omega_{{\pmb{v}} }$, where $\Omega_x\subset\mathbb{R}^1$, $\Omega_{{\pmb{v}} }:=\Omega_{v_1}\otimes\Omega_{v_2}\subset\mathbb{R}^2$. The electromagnetic fields are given by
\begin{equation*}
    \pmb{E}(x,t) = (E_1(x,t),E_2(x,t),0),\qquad\pmb{B}(x,t) = (0,0,B_3(x,t)). 
\end{equation*}
The advection field of the Vlasov equation becomes
\[
    \pmb{\bbetaa} = \begin{pmatrix}
        v_1,E_1+v_2B_3,E_2-v_1B_3
    \end{pmatrix}^\top,\qquad\nabla = \begin{pmatrix}
        \partial_x,\partial_{v_1},\partial_{v_2}
    \end{pmatrix}^\top.
\]
The coupled Maxwell's equations are given by
\begin{equation}
\begin{aligned}
  \partial_t E_1&=-J_1,\\
  \partial_t E_2&=-\partial_x B_3-J_2,\\
  \partial_t B_3&= -\partial_x E_2 ,\\
  \partial_x E_1&=\rho-\rho_0.
\end{aligned}\notag
\end{equation}
In this simplified setting, the system \eqref{eq:system} for the Vlasov--Maxwell equations reduces to:
\begin{equation}
\begin{aligned}
    (\polM^x\otimes \polM^\bv)\dot{\polf} + 
    \Big( 
      \polA^{x,1}\otimes\polC^{\bv,1} 
      & + 
       \sum_{l=1}^2\polC^{x,l}(\bE)\otimes \polA^{\bv,l} + 
      \polC^{x,3}(\bB)\otimes \polG^{\bv,3}
    \Big)  
      \polf\\
    & +
     (\polD^x\otimes \polM^\bv+\polM^x\otimes \polD^\bv)\polf= 0,
  \end{aligned}\notag
\end{equation}

Next, let $\Omega := \Omega_{\pmb{x}}\otimes\Omega_{\pmb{v}}$, where $\Omega_{\pmb{x}}:=\Omega_{x_1}\otimes\Omega_{x_2}\subset\mathbb{R}^2$, $\Omega_{{\pmb{v}} }:=\Omega_{v_1}\otimes\Omega_{v_2}\subset\mathbb{R}^2$. The fields are then simplified as follows
\begin{equation*}
    \pmb{E}({\pmb{x}},t) = (E_1({\pmb{x}},t),E_2({\pmb{x}},t),0),\qquad\pmb{B}({\pmb{x}},t) = (0,0,B_3({\pmb{x}},t)). 
\end{equation*}
The electromagnetic fields are given by
\[
    \pmb{\bbetaa} = \begin{pmatrix}
        v_1,v_2,E_1+v_2B_3,E_2-v_1B_3
    \end{pmatrix}^\top,\qquad\nabla = \begin{pmatrix}
        \partial_{x_1},\partial_{x_2},\partial_{v_1},\partial_{v_2}
    \end{pmatrix}^\top.
\]
The coupled Maxwell's equations are given by
\begin{equation}
\begin{aligned}
  \partial_t E_1&=\partial_{x_2}B_3-J_1,\\
  \partial_t E_2&=-\partial_{x_1} B_3-J_2,\\
  \partial_t B_3&= -\partial_{x_1} E_2+\partial_{x_2}E_1 ,\\
  \partial_{x_1} E_1+\partial_{x_2} E_2&=\rho-\rho_0.
\end{aligned}\notag
\end{equation}
Similarly, in the 2D2V phase space, system \eqref{eq:system} becomes
\begin{equation}
\begin{aligned}
    (\polM^\bx\otimes \polM^\bv)\dot{\polf} + 
    \Big( 
      \sum_{l=1}^2\polA^{\bx,l}\otimes\polC^{\bv,l} 
      & + 
       \sum_{l=1}^2\polC^{\bx,l}(\bE)\otimes \polA^{\bv,l} + 
      \polC^{\bx,3}(\bB)\otimes \polG^{\bv,3}
    \Big)  
      \polf\\
    & +
     (\polD^\bx\otimes \polM^\bv+\polM^\bx\otimes \polD^\bv)\polf= 0.
  \end{aligned}\notag
\end{equation}

\begin{remark}
Note that in both reduced models, only the third component of the magnetic field is nonzero and there is no dependence on $x_3$. Therefore, the magnetic Gauss's law is trivially satisfied.
\end{remark}

\subsection{Landau damping}
First, we study the electrostatic example of strong Landau damping in 1D2V to verify the necessity of using artificial viscosity. We take the following initial distribution
\begin{equation}
  f_0(x,{\pmb{v}}) = \frac{1}{2\pi}{\rm exp}\left(-\frac{v_1^2+v_2^2}{2}\right)(1+\alpha{\rm cos}(\theta x)),\notag
\end{equation}
where $\alpha = 0.5$ and $\theta=0.5$, in the \textcolor{black}{computational} phase space \textcolor{black}{domain} $\Omega = [0,2\pi/\theta]\otimes[-5,5]^2$. The fields $E_2(x,t)$ and $B_3(x,t)$ are set to be $0$ during the whole simulation; we obtain $E_1(x,0)$ by solving the Poisson equation \eqref{eq:poisson} and updating $E_1$ with the Ampère's equation \eqref{eq:AP}. We apply both the standard Galerkin finite elements and viscous regularization methods, and solve the system on a mesh consisting of $31\times61^2$ nodes; we plot the time evolution of the electric energy $\calE_1=\frac{1}{2}\int_{\Omega_x}(E_1)^2\ {\rm d}x$ on a logarithmic scale in Figure~\ref{fig:LD}.
\begin{figure}[htbp]
  \centering
  \begin{subfigure}[b]{.5\textwidth}
    \includegraphics[width=\linewidth]{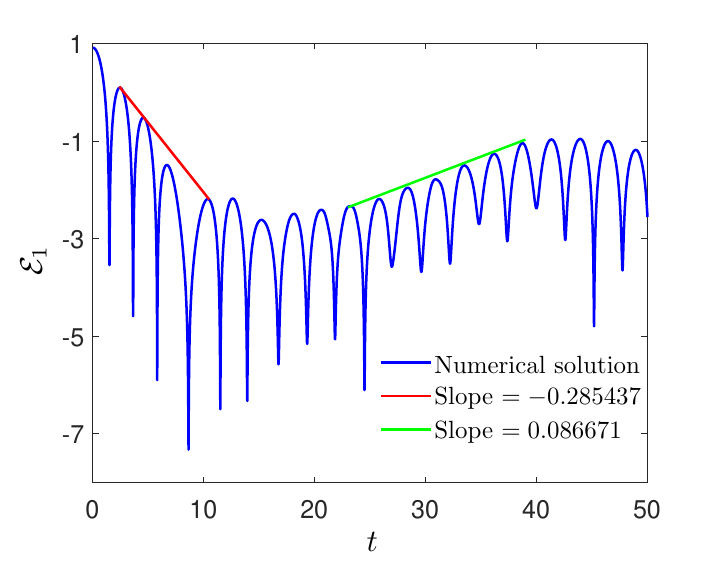}
    \caption{$\polQ_1$ viscous regularization solution}
  \end{subfigure}\hfill
  \centering
  \begin{subfigure}[b]{.5\textwidth}
    \includegraphics[width=\linewidth]{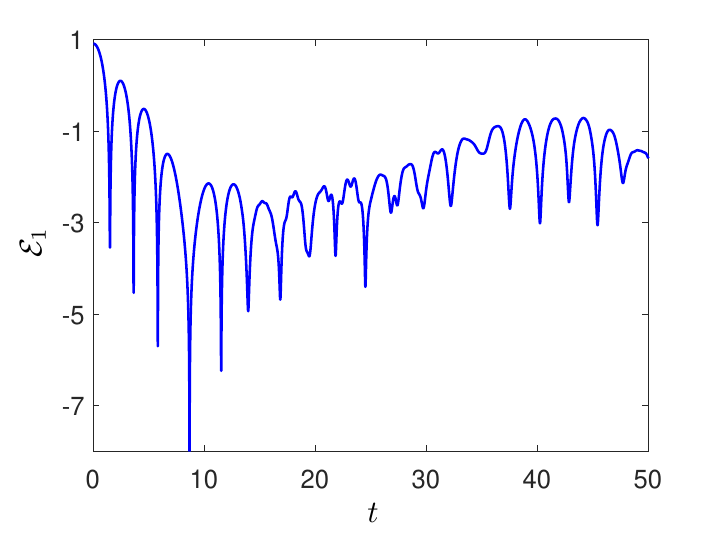}
    \caption{$\polQ_1$ Galerkin solution}
  \end{subfigure}\hfill
  \centering
  \begin{subfigure}[b]{.5\textwidth}
    \includegraphics[width=\linewidth]{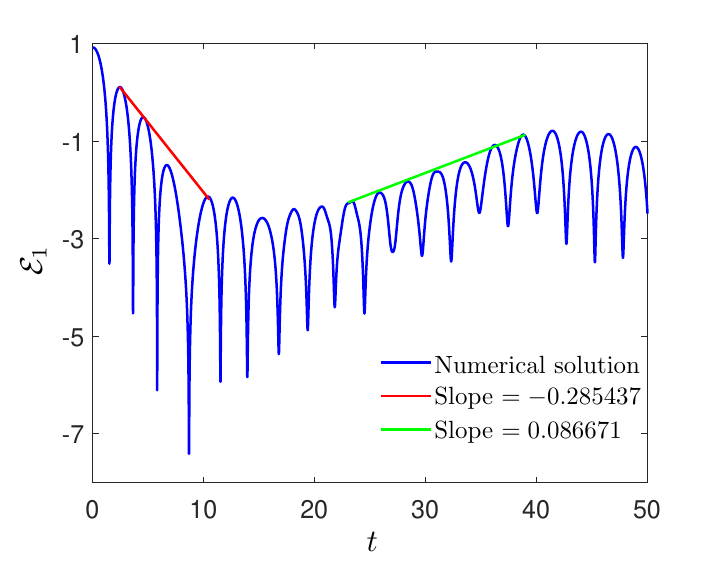}
    \caption{$\polQ_2$ viscous regularization solution}
  \end{subfigure}\hfill
  \centering
  \begin{subfigure}[b]{.5\textwidth}
    \includegraphics[width=\linewidth]{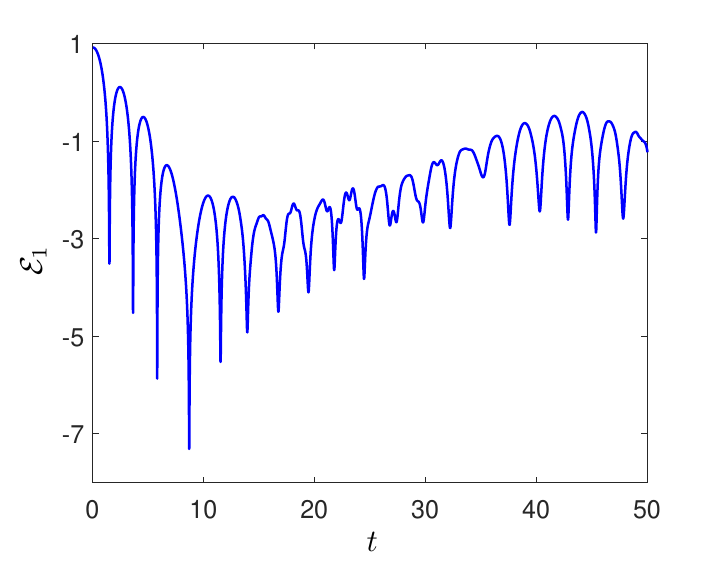}
    \caption{$\polQ_2$ Galerkin solution}
  \end{subfigure}\hfill
  \begin{subfigure}[b]{.5\textwidth}
    \includegraphics[width=\linewidth]{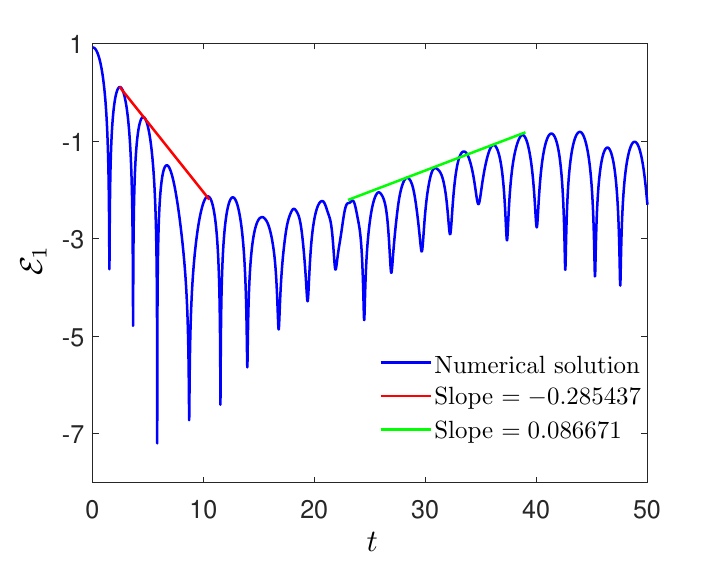}
    \caption{$\polQ_3$ viscous regularization solution}
  \end{subfigure}\hfill
  \centering
  \begin{subfigure}[b]{.5\textwidth}
    \includegraphics[width=\linewidth]{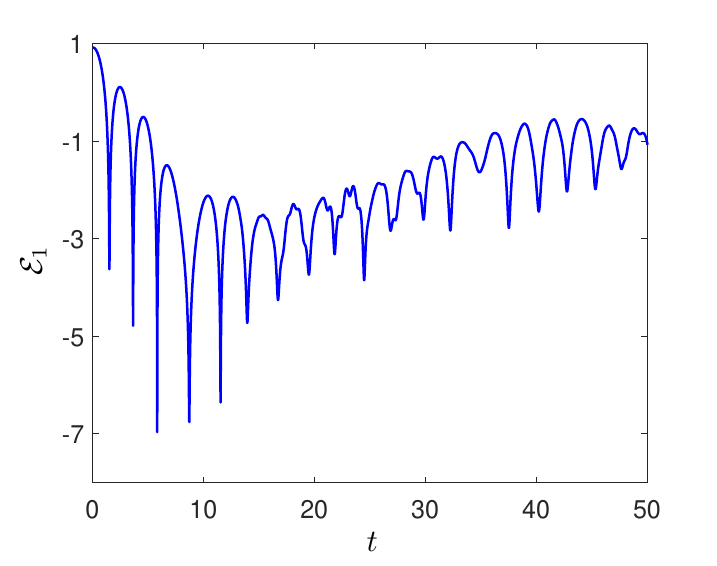}
    \caption{$\polQ_3$ Galerkin solution}
  \end{subfigure}\hfill
  \caption{ Landau damping: electric energy with fitted damping and growth rates. The number of degrees of freedoms is $31\times61^2$ for all the results presented here.}
  \label{fig:LD}
\end{figure}

On the one hand, we observe spurious oscillations in the solutions obtained with Galerkin's method for all polynomial degrees, as shown in Figure~\ref{fig:LD} (b), (d), and (f). These spurious oscillations are expected, as strong filamentation occurs in the solutions of Landau damping when $\alpha = 0.5$. This issue was observed in other discretization schemes, for instance, \cite[Fig. 4]{MR1977366}, where the authors compared several schemes to solve the Landau damping problem. On the other hand, the viscous regularization method yields stable solutions for all polynomial spaces. We plot the results of the regularized scheme in Figure~\ref{fig:LD} (a), (c), and (e). We fit the damping and growth rates of the electric energy. These results align closely with the growth rates reported in \cite[Table~3]{Kraus_Kormann_Morrison_Sonnendrücker_2017}, regardless of the polynomial degree. The artificial viscosity introduced in this article serves the same purpose of stabilizing the solutions as the positive and flux-conservative (PFC) technique described in \cite{filbet2001}. 

Next, we study the linear Landau damping in the 2D2V setting, using the same parameters as in reference \cite{ricketson2024explicitenergyconservingparticleincellscheme}. We take the following initial distribution
\begin{equation}
  f_0({\pmb{x}},{\pmb{v}}) = \frac{1}{2\pi}{\rm exp}\left(-\frac{v_1^2+v_2^2}{2}\right)(1+\alpha{\rm cos}(\theta x_1))(1+\alpha{\rm cos}(\theta x_2)),\notag
\end{equation}
in the phase space $\Omega = [0,L]^2\otimes[-5,5]^2$, where $L=22$, $\alpha = 0.05$, and $\theta=2\pi/L$. According to \cite[Section 5.1]{ricketson2024explicitenergyconservingparticleincellscheme}, the theoretical damping rate of the electric energy is given by
\begin{equation}
  \gamma\approx\sqrt{\frac{\pi}{2}}\frac{\omega^2}{2\theta^3}\exp\left(\frac{-\omega^2}{2\theta^2}\right),\qquad\omega^2=1+3\theta^2.\notag
\end{equation}
We plot $\calE=\frac{1}{2}\int_{\Omega_{\pmb{x}}}\pmb{E}^2\ {\rm d}{\pmb{x}}$ and the reference damping rate in Figure~\ref{fig:2D2Vld} (a); the solution is computed using the viscous regularization method with $\polQ_1$ elements, and the number of degrees of freedom is $33^2\times65^2$.
Our numerical results in the high-dimensional case show good agreement with the theoretical damping rate. We also verify whether Gauss' law is satisfied. Recalling \eqref{eq:weak:GL}, our methods preserve the Gauss' law in the weak form, i.e.,$\big(\bE^{n}_h,\, \GRAD_{{\pmb{x}}}\phi_i\big)=-\big(\rho^{n}_h,\,\phi_i\big), \forall i \in \calI_{\bx}$. To quantify the error in Gauss' law, we define the following measurement for the error of the Gauss' law:
\begin{equation}
  {\rm error} = \sqrt{\sum_{i=1}^{N_\bx}\Big(\big(\bE^{n}_h,\, \GRAD_{{\pmb{x}}}\phi_i\big)+\big(\rho^{n}_h,\,\phi_i\big)\Big)^2}\notag.
\end{equation}
The error of Gauss' law is shown in Figure \ref{fig:2D2Vld} (b). As observed, it resembles random noise, which can be attributed to round-off errors.

\begin{figure}[htbp]
  \centering
  \begin{subfigure}[b]{0.5\textwidth}
    \includegraphics[width=\linewidth]{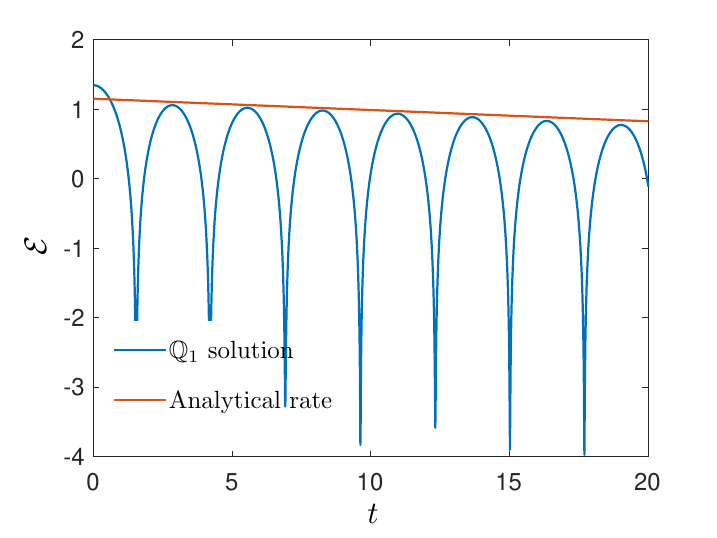}
      \caption{$\calE$}
  \end{subfigure}\hfill
  \begin{subfigure}[b]{0.5\textwidth}
    \includegraphics[width=\linewidth]{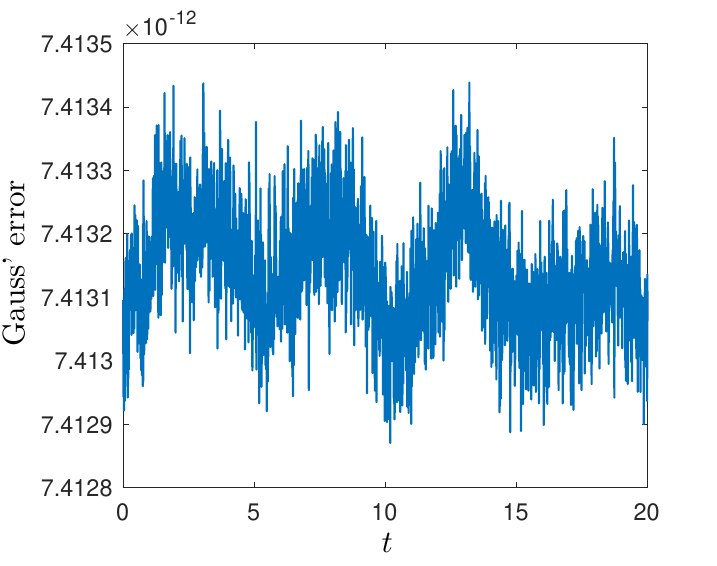}
      \caption{Gauss' error}
  \end{subfigure}
  \hfill
  \caption{Landau damping: time evolution of $\calE$ (a) and Gauss' error (b). The solution is obtained using $\polQ_1$ elements, and the number of degrees of freedoms is $33^2\times65^2$.}
  \label{fig:2D2Vld}
\end{figure}

\subsection{Weibel instability}
In this example, we apply the proposed method to solve the so-called Weibel instability. We use the same initial settings as in \cite{MR3294671}. The initial distribution and fields are given by:
\begin{equation*}
    \begin{aligned}
  f_0(x,{\pmb{v}} )&= \frac{1}{2\pi\sigma_1\sigma_2}{\rm exp}\left(-\frac{1}{2}\left(\frac{v_1^2}{\sigma_1^2}+\frac{v_2^2}{\sigma_2^2}\right)\right)(1+\alpha{\rm cos}(\theta x)),\\
  B_{0,3}(x) &= \beta{\rm cos}(\theta x),\\
  E_{0,2}(x) &= 0,
    \end{aligned}
\end{equation*}
and $E_{0,1}(x)$ is computed from the Poisson equation \eqref{eq:poisson}. We use the following parameters in the numerical simulations: $\sigma_1 = 0.02/\sqrt{2}$, $\sigma_2 = \sqrt{12}\sigma_1$, $\theta = 1.25$, $\alpha = 10^{-4}$, and $\beta = 10^{-4}$, and the \textcolor{black}{computational} phase space \textcolor{black}{domain} is $\Omega=[0,2\pi/\theta]\otimes[-5\sigma_1,5\sigma_1]\otimes[-5\sigma_2,5\sigma_2]$.

This test is commonly used to study the accuracy of numerical schemes for the Vlasov--Maxwell system. Following the approach described in \cite[Sec. 4.1]{MR3267101}, we derive analytical solutions to the Vlasov--Maxwell equations for specific cases. The procedure is as follows: starting with the initial data $f({\pmb{x}} , {\pmb{v}} , 0)$, $\pmb{E}({\pmb{x}} , 0)$, and $\pmb{B}({\pmb{x}} , 0)$, we solve the system to obtain the solutions $f({\pmb{x}} , {\pmb{v}} , T)$, $\pmb{E}({\pmb{x}} , T)$, and $\pmb{B}({\pmb{x}} , T)$ at $t = T$. Then, using $f({\pmb{x}} , -{\pmb{v}} , T)$, $\pmb{E}({\pmb{x}} , T)$, and $-\pmb{B}({\pmb{x}} , T)$ as new initial data, we solve the system again. The analytical solutions at $t = T$ should be $f({\pmb{x}} , -{\pmb{v}} , 0)$, $\pmb{E}({\pmb{x}} , 0)$, and $-\pmb{B}({\pmb{x}} , 0)$. Note that in this case, $E_2=0$, so the convergence orders concerning $E_2$ are not presented.

We solve the problem until $T=5$ for different mesh resolutions. In Table~\ref{tab:WIRV}, we present the $L^2$-norms of the errors and corresponding convergence rates for $f$, $E_1$, and $B_3$. We observe the optimal convergence rates of $\calO(h^{k+1})$ for the function $f$, which correspond to the theoretical rates for continuous finite element methods in smooth problems. Additionally, we observe that the convergence rate for $E_1$ and $B_3$ is $\calO(h^{k})$; this is because the polynomial degree for the function spaces of $E_1$ and $B_3$ is $k-1$. 
\begin{table}[htbp]
  \caption{Weibel instability: $L^2$-errors and convergence orders of viscous regularization solutions.}  \label{tab:WIRV}
\begin{center}
  \begin{tabular}{lllllllll} \hline
    \multicolumn{2}{l}{} &\multicolumn{5}{l}{\#DOFs} \\ \cline{3-9}
    \multicolumn{2}{l}{} &$31^3$ & &$61^3$  & & &$121^3$ & \\ \cline{3-3}\cline{5-6}\cline{8-9}
    \multicolumn{2}{l}{} &$L^2$-errors & &$L^2$-errors  &orders & &$L^2$-errors &orders \\ \hline
    \multirow{3}{*}{$k=1$} &$f$ &1.40E-2&&3.51E-03 &2.00 & &8.78E-04 &2.00\\
    &$E_1$ &6.09E-2&&3.03E-02 &1.01 & &1.51E-02 &1.00\\
    &$B_3$ &6.04E-2&&3.02E-02 &1.00 & &1.51E-02 &1.00\\ \hline
    \multirow{3}{*}{$k=2$} &$f$ &3.24E-3&&4.10E-04 &2.98 & &5.15E-05 &2.99\\
    &$E_1$ &6.52E-3&&1.63E-03 &2.00 & &4.09E-04 &2.01\\
    &$B_3$ &1.62E-2&&4.02E-03 &2.00 & &1.00E-03 &2.00\\ \hline
    \multirow{3}{*}{$k=3$} &$f$ &1.06E-3&&7.21E-05 &3.87 & &4.55E-06 &3.99\\
    &$E_1$ &7.78E-4&&9.75E-05 &3.00 & &1.22E-05 &2.99\\
    &$B_3$ &1.42E-3&&1.78E-04 &2.99 & &2.23E-05 &3.00\\ \hline
  \end{tabular}
\end{center}
\end{table}

In addition, we are interested in the behavior of the electromagnetic energy. We depict the time evolution of $\frac{1}{2}\Vert E_1\Vert^2$, $\frac{1}{2}\Vert E_2\Vert^2$, and $\frac{1}{2}\Vert B_3\Vert^2$ obtained using the viscous regularization method in Figure~\ref{fig:WI}. The green line in the plots shows the growth rate of 0.02784 which can be found from a linearized analysis (see \cite[Sec 4.1]{Kraus_Kormann_Morrison_Sonnendrücker_2017}). One can see, that this growth rate is confirmed in our results. 
\begin{figure}[htbp]
  \centering
  \begin{subfigure}[b]{0.5\textwidth}
    \includegraphics[width=\linewidth]{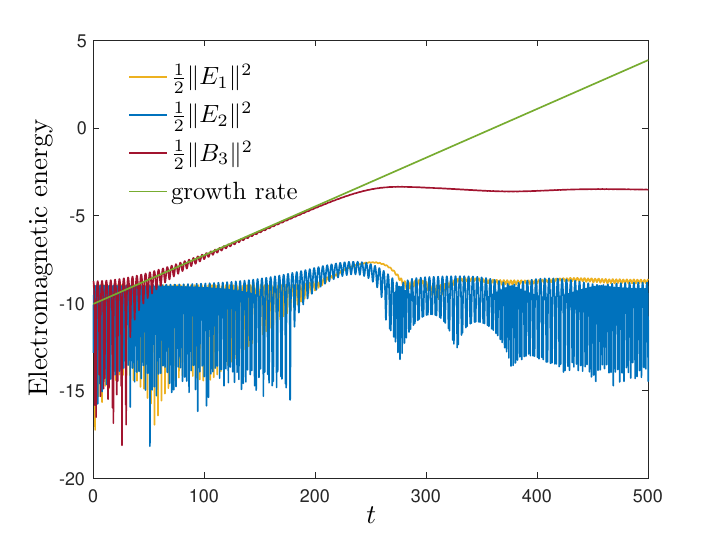}
    \caption{$\polQ_1$}
  \end{subfigure}\hfill
  \begin{subfigure}[b]{0.5\textwidth}
    \includegraphics[width=\linewidth]{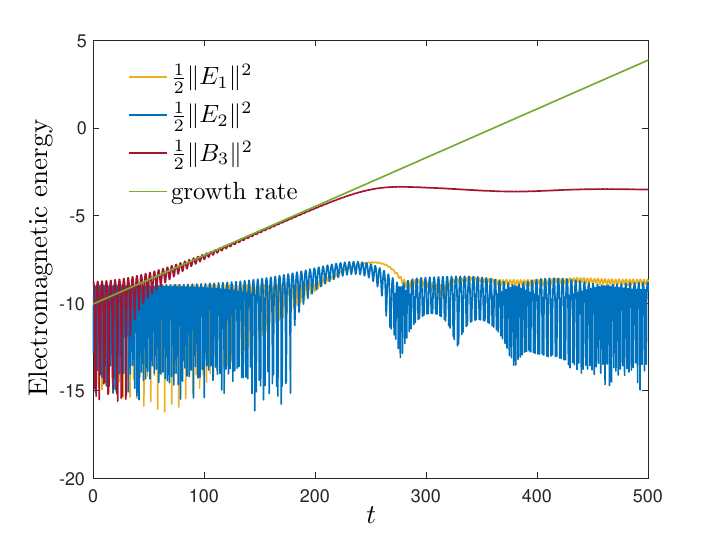}
    \caption{$\polQ_2$}
  \end{subfigure}\hfill
  \begin{subfigure}[b]{0.5\textwidth}
    \includegraphics[width=\linewidth]{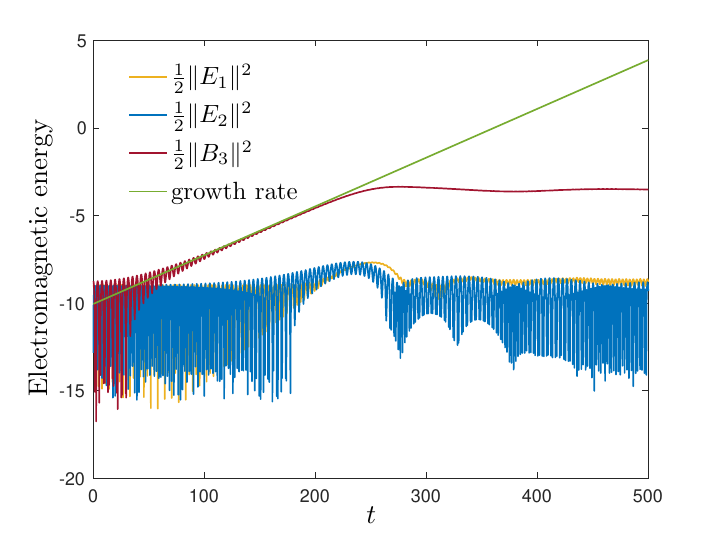}
    \caption{$\polQ_3$}
  \end{subfigure}\hfill
  \caption{Weibel instability: the two electric and the magnetic energy together with the analytic growth rate. The number of degrees of freedoms is $31\times61^2$.}
  \label{fig:WI}
\end{figure}
In the above analysis, we have proved that the proposed method preserves the total mass and Gauss' law. To verify these properties numerically, we calculate the maximum deviation from gauss' law and the relative error of the total mass. The results are collected in Table~\ref{tab:divWI}. It can be seen that the deviation of the total mass and electric Gauss' law is close to machine precision.
\begin{table}[htbp]
    \caption{Weibel instability: maximum deviation of total mass and Gauss' law until $t=500$. The number of degrees of freedom is $31\times61^2$.} 
    \label{tab:divWI}  
    \begin{center}
    \begin{tabular}{c|c|c} \hline
      &Total mass&Gauss' law\\ \hline
     $\polQ_1$&3.6E-13&1.3E-14\\
     $\polQ_2$&7.8E-14&8.7E-15\\
     $\polQ_3$&7.9E-14&8.3E-15\\ \hline
    \end{tabular}
  \end{center}
\end{table}

\subsection{Streaming Weibel instability}
We consider the streaming Weibel instability. We use the same initial distribution and fields as described in \cite{cheng2014discontinuous}:
\begin{equation}\notag
\begin{aligned}
     f_0(x,{\pmb{v}} ) =& \frac{1}{2\pi\sigma^2}{\rm exp}\left(-\frac{v_1^2}{2\sigma^2}\right)\left(\delta{\rm exp}\left(-\frac{(v_2-v_{0,1})^2}{2\sigma^2}\right)\right.\\
     &\left.+(1-\delta){\rm exp}\left(-\frac{(v_2-v_{0,2})^2}{2\sigma^2}\right)\right),\\
      B_{0,3}(x) = &\beta \sin(\theta x),\\
      E_{0,2}(x) = &0,
\end{aligned}
\end{equation}
and the value of $E_{0,1}(x)$ is computed from the Poisson equation. The following parameters are used: $\sigma = 0.1/\sqrt{2}$, $\theta = 0.2$, $\beta = 10^{-3}$, $v_{0,1}=0.5$, $v_{0,2}=-0.1$ and $\delta=1/6$. The \textcolor{black}{computational} phase space \textcolor{black}{domain} is $\Omega=[0,2\pi/\theta]\otimes[-5\sigma,5\sigma]\otimes[-1.2,1.2]$.

For this set of parameters, the growth rate of the electric energy $\calE_2$ is found to be $0.03$ in \cite{califano1997spatial}. Figure \ref{fig:SWI} shows that the results are consistent with this growth rate. As in the previous example, we calculate the relative errors of the total mass and the electric Gauss' law. The result, presented in table~\ref{tab:divSWI}, shows that mass conservation and Gauss' law are satisfied for this problem as well.
\begin{figure}[htbp]
  \centering
  \begin{subfigure}[b]{0.5\textwidth}
    \includegraphics[width=\linewidth]{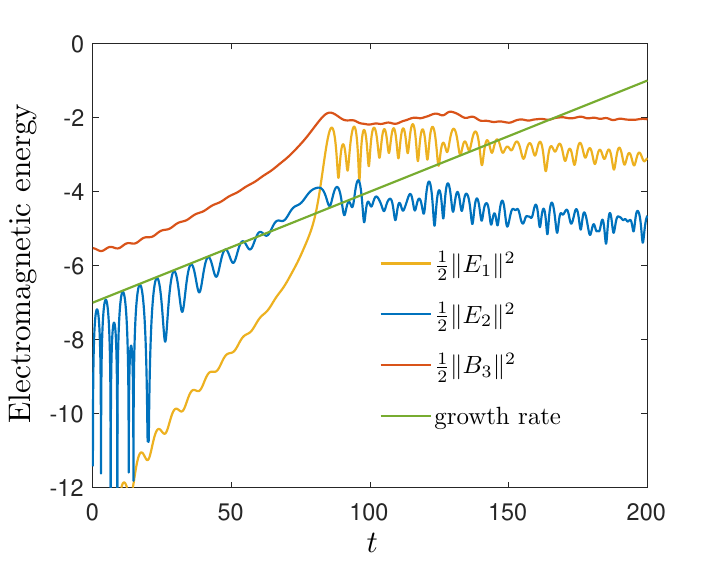}
    \caption{$\polQ_1$}
  \end{subfigure}\hfill
  \begin{subfigure}[b]{0.5\textwidth}
    \includegraphics[width=\linewidth]{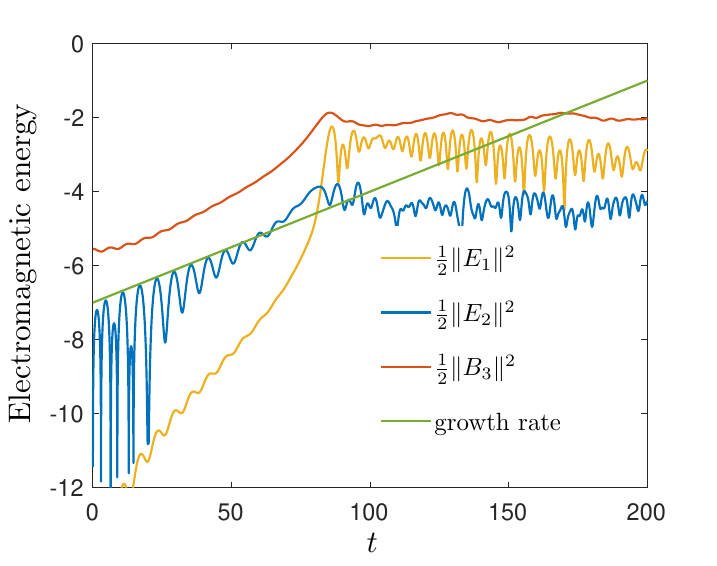}
    \caption{$\polQ_2$}
  \end{subfigure}\hfill
  \begin{subfigure}[b]{0.5\textwidth}
    \includegraphics[width=\linewidth]{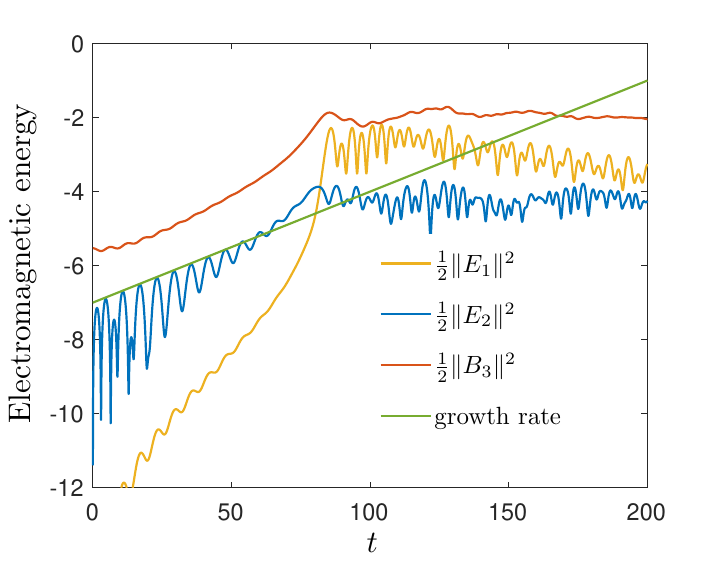}
    \caption{$\polQ_3$}
  \end{subfigure}\hfill
  \caption{Streaming Weibel instability: the two electric and the magnetic energy together with the analytic growth rate. The number of degrees of freedoms is $31\times61^2$.}
  \label{fig:SWI}
\end{figure}
In Figure~\ref{fig:SWIsamedof}, we plot the deviations of the total energy and the squared $L^2$-norm for different polynomial spaces using the same number of degrees of freedom, and very similar results are obtained for all polynomial spaces. In Figure~\ref{fig:SWIsameele}, we fix the polynomial degree, in this case, we use $\polQ_1$ and compute the deviations of the total energy and squared $L^2$-norm. We observe that the deviations decrease as the mesh is refined.
\begin{table}[htbp] 
  \caption{Streaming Weibel instability: maximum deviation of total mass and Gauss' law until $t=200$. The number of degrees of freedom is $31\times61^2$.}
  \label{tab:divSWI}  
  \begin{center}
  \begin{tabular}{c|c|c} \hline
    &Total mass&Gauss' law\\ \hline
   $\polQ_1$&9.3E-13&2.7E-13\\
   $\polQ_2$&1.8E-13&2.7E-14\\
   $\polQ_3$&1.7E-13&1.0E-13\\ \hline
  \end{tabular}
\end{center}
\end{table}

\begin{figure}[htbp]
  \centering
  \begin{subfigure}[b]{0.5\textwidth}
    \includegraphics[width=\linewidth]{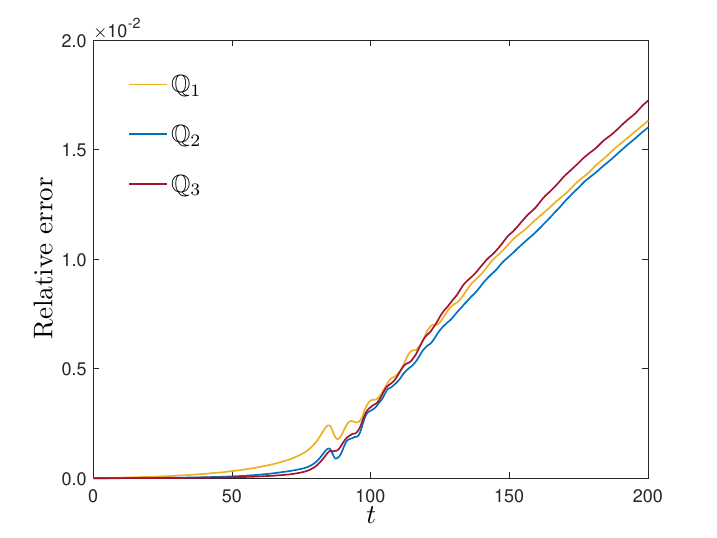}
    \caption{Total energy}
  \end{subfigure}\hfill
  \begin{subfigure}[b]{0.5\textwidth}
    \includegraphics[width=\linewidth]{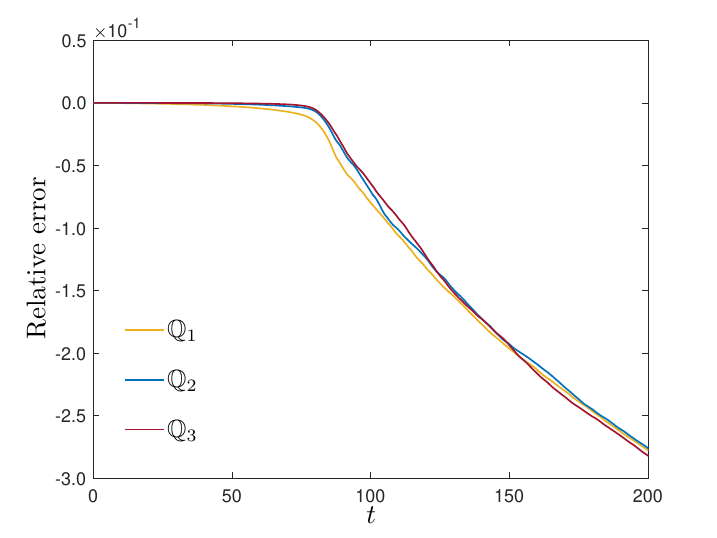}
    \caption{Squared $L^2$-norm}
  \end{subfigure}\hfill
  \caption{Streaming Weibel instability: the deviations of total energy and squared $L^2$ norm when the solutions are obtained using $\polQ_1$, $\polQ_2$, and $\polQ_3$ elements. The number of degrees of freedoms is $31\times61^2$.}
  \label{fig:SWIsamedof}  
\end{figure}
\begin{figure}[htbp]
  \centering
  \begin{subfigure}[b]{0.5\textwidth}
    \includegraphics[width=\linewidth]{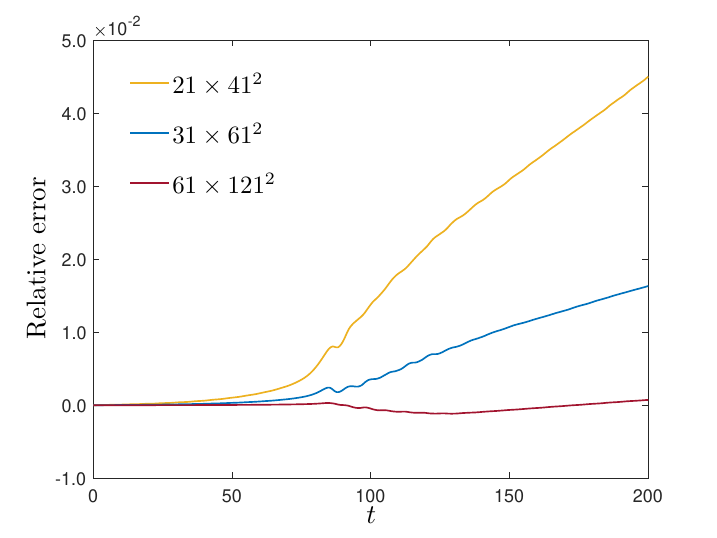}
    \caption{Total energy}
  \end{subfigure}\hfill
  \begin{subfigure}[b]{0.5\textwidth}
    \includegraphics[width=\linewidth]{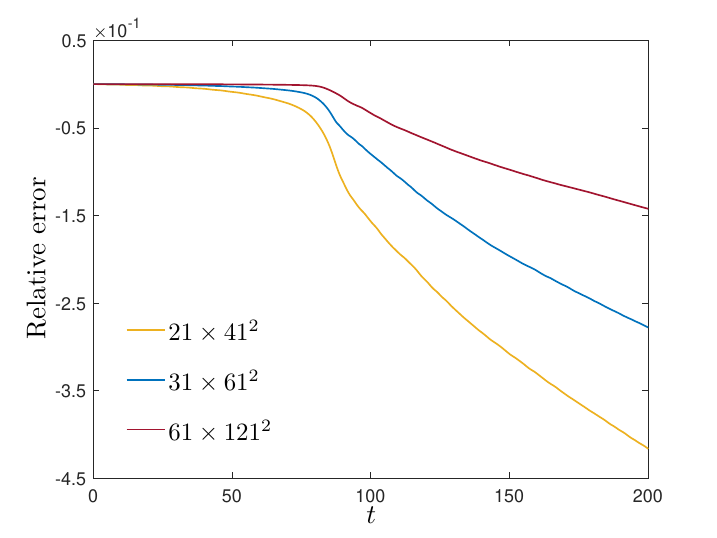}
    \caption{Squared $L^2$-norm}
  \end{subfigure}\hfill
  \caption{Streaming Weibel instability: the deviations of total energy and squared $L^2$-norm when the solutions are obtained in different meshes. The element is $\polQ_1$.}
  \label{fig:SWIsameele}  
\end{figure}

\subsection{Two-stream instability} 
Now we investigate is the two-stream instability, which is given using the following initial value
\begin{equation}\notag
\begin{aligned}
     f_0(x,{\pmb{v}} )=&\frac{1}{4\pi}\left({\rm exp}\left(-\frac{(v_1-2.4)^2}{2}\right)+{\rm exp}\left(-\frac{(v_1+2.4)^2}{2}\right)\right)\\
    &{\rm exp}\left(-\frac{v_2^2}{2}\right)(1+\alpha {\rm cos}(\theta x)),\\
     B_{0,3}(x) &= 0,\\
     E_{0,2}(x) &= 0,
\end{aligned}
\end{equation}
where $\alpha = 10^{-3}$, $\theta=0.2$, and $E_{0,1}(x)$ is computed from the Poisson equation \eqref{eq:poisson}. The \textcolor{black}{computational} phase space \textcolor{black}{domain} is $\Omega=[0,2\pi/\theta]\otimes[-7.5,7.5]\otimes[-5.0,5.0]$.

We plot the time evolution of the electric energy in Figure~\ref{fig:TS}. The left panel of the figure presents the solution at an earlier time $T=50$ for three different polynomial spaces using the same degrees of freedom. Our findings align well with those reported in \cite[Fig 2]{MR4402737}, where comparable results were achieved using a Fourier-based approach. We are also interested in seeing the long-term behavior of the two-stream instability. We plot the solutions until $T=200$ in the right panel of Figure~\ref{fig:TS}. We observe that the electric energy remains constant after a short period of oscillation, and there is no noticeable loss or gain of energy. 
\begin{figure}[htbp]
  \centering
  \begin{subfigure}[b]{0.5\textwidth}
    \includegraphics[width=\linewidth]{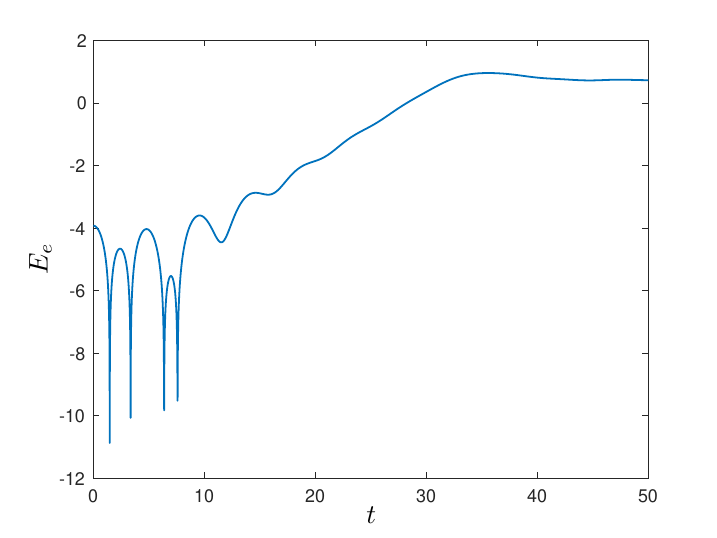}
    \caption{$\polQ_1, T=50$}
  \end{subfigure}\hfill
  \centering
  \begin{subfigure}[b]{0.5\textwidth}
    \includegraphics[width=\linewidth]{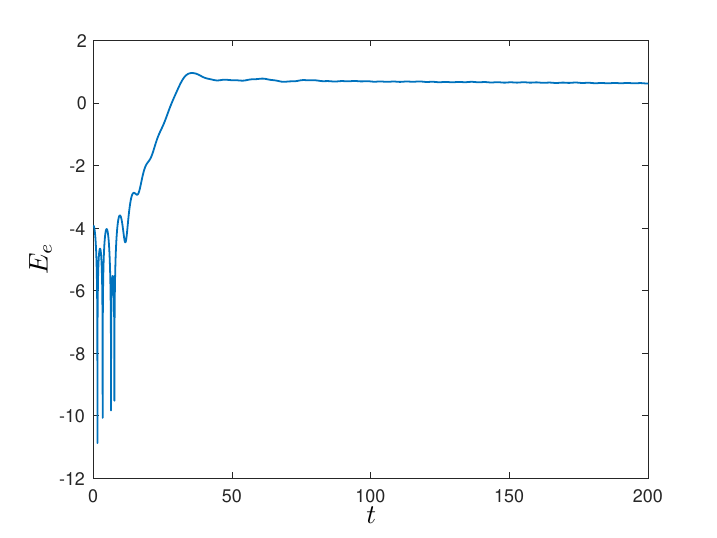}
    \caption{$\polQ_1, T=200$}
  \end{subfigure}\hfill
  \centering
  \begin{subfigure}[b]{0.5\textwidth}
    \includegraphics[width=\linewidth]{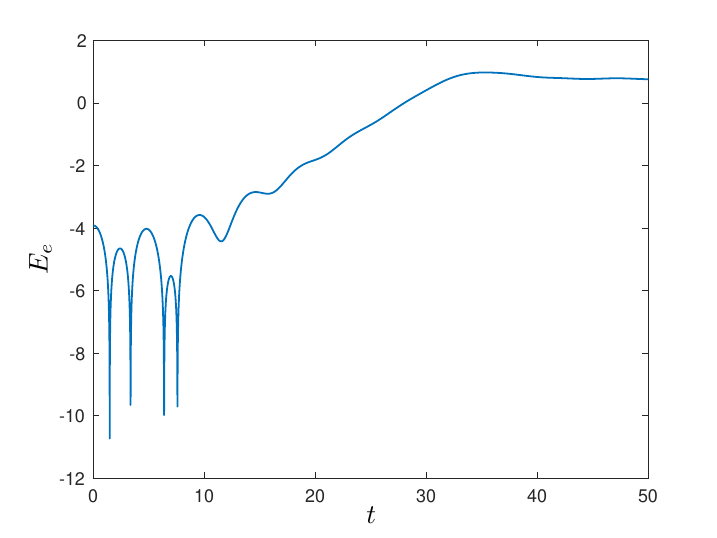}
    \caption{$\polQ_2, T=50$}
  \end{subfigure}\hfill
  \centering
  \begin{subfigure}[b]{0.5\textwidth}
    \includegraphics[width=\linewidth]{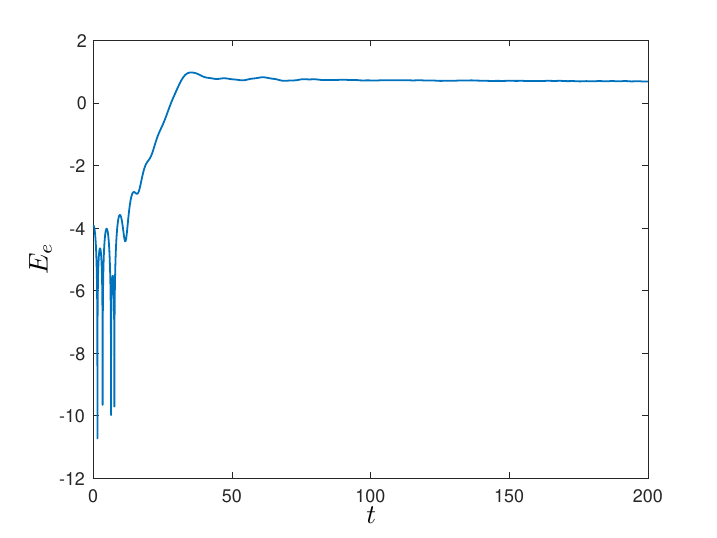}
    \caption{$\polQ_2, T=200$}
  \end{subfigure}\hfill
  \centering
  \begin{subfigure}[b]{0.5\textwidth}
    \includegraphics[width=\linewidth]{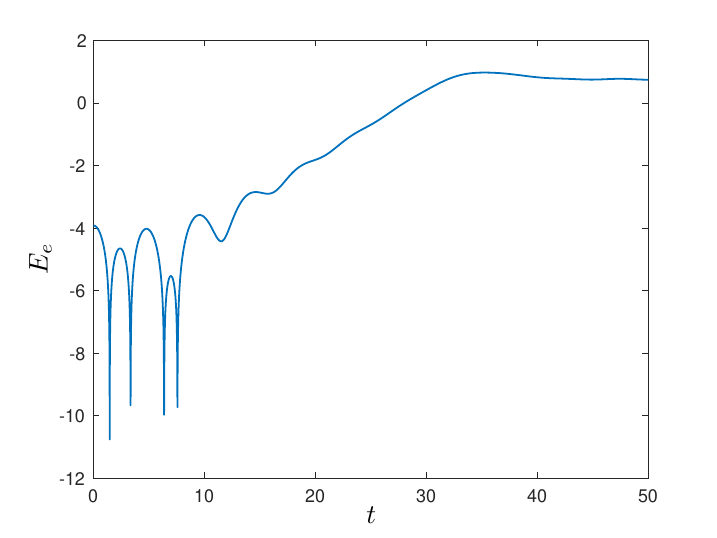}
    \caption{$\polQ_3, T=50$}
  \end{subfigure}\hfill
  \centering
  \begin{subfigure}[b]{0.5\textwidth}
    \includegraphics[width=\linewidth]{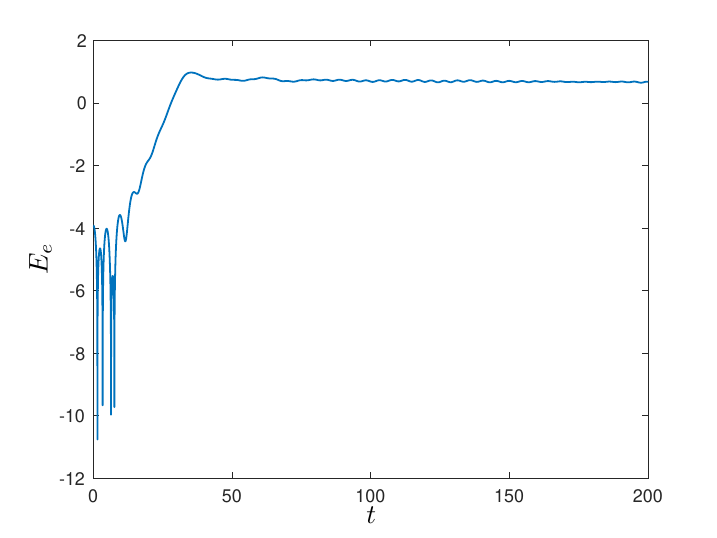}
    \caption{$\polQ_3, T=200$}
  \end{subfigure}\hfill
  \caption{Two-stream instability: time evolution of the ﬁrst component of the electric energy. The results are obtained in the meshes consist of $31\times61^2$ degrees of freedom.}
  \label{fig:TS}
\end{figure}

\subsection{Diocotron instability}
Finally, we apply our methods to a more complex geometric setting. The $\polP_k$ element is another commonly used element in finite element methods. One key advantage of $\polP_k$ element is that it employs simplices as reference cells, enabling it to conform naturally to arbitrary geometries and complex boundaries, while $\polQ_k$ elements are typically limited to structured grid meshes. When necessary—for instance, when the physical space has an unstructured shape—we can employ $\polP_k$ elements. As mentioned earlier, our methods utilize tensor products, with the meshes and function spaces for $\Omega_{\pmb{x}}$ and $\Omega_{\pmb{v}}$ defined independently. This flexibility allows us to use a combination of $\polP_k \otimes \polQ_k$ elements.

We consider the diocotron instability in the system of 2D2V Vlasov--Maxwell. This benchmark case has been discussed in many reference, see \cite{MR3267094,MR3485969,MR4456467} for instance. The initial distribution function is given by
\begin{equation}
  f_0({\pmb{x}},{\pmb{v}})= \frac{\rho_0({\pmb{x}})}{2\pi}\exp\left(-\frac{v_1^2+v_2^2}{2}\right),\notag
\end{equation}
where 
\begin{equation}
  \rho_0({\pmb{x}})= \left\{
    \begin{aligned}
      &(1+\alpha\cos(\ell\theta))\exp(-4(\Vert{\pmb{x}}\Vert-6.5)^2),\quad&&{\rm if}\ 5\leq \Vert{\pmb{x}}\Vert\leq 8, \\
      &0, &&{\rm otherwise,}
    \end{aligned}
  \right.\notag
\end{equation}
and $\alpha = 0.2$, $\theta = {\rm atan}(x_2/x_1)$ and $\ell=6$ is the number of vortices. The initial electric field $\pmb{E}=(E_1,E_2,0)^{\mathsf{T}}$ is obtained by solving the Poisson equation \eqref{eq:poisson} and we have a initial magnetic field whose value is $\pmb{B}=(0,0,1)^{\mathsf{T}}$. 

In $\Omega_{\pmb{x}}$ we use the $\polP_1$ element with 10000 degrees of freedom, and apply 0 Dirichlet boundary condition. Since $\calT_\bx$ is no longer a structured grid mesh, we define its mesh size as a function $h(\bx)\in\polV_\bx$ using the following projection: find $h(\bx)\in\polV_\bx$ such that
\begin{equation}
  \sum_{K\in\calT_\bx}\int_K h(\bx)\phi_i\ {\rm d}\bx=\sum_{K\in\calT_\bx}\int_K \frac{r}{k}\phi_i\ {\rm d}\bx,\qquad\forall i\in\calI_\bx,\notag
\end{equation}
where $r$ is the circumradius of the element $K$. In $\Omega_{\pmb{v}}$ we use the $\polQ_1$ element with $20^2$ degrees of freedom and periodic boundary condition. The  charge density $\rho_h$ at $t=0$ and $t=1$ are shown in Figure \ref{fig:DI1}, where we observe the spreading of the density over time. In Figure \ref{fig:DI2}, we plot the distributions of the residual and artificial viscosity in the ${\pmb{x}}$-direction at $t=1$. It can be seen that the residual follows the propagation of $\rho_h$, and the presence of a six-petal shape reflects the vortex number of the initial density. In addition, the artificial viscosity has the same distribution as the residual, which means the residual-based viscosity is most often activated in place of the first-order viscosity.
\begin{figure}[htbp]
  \centering
  \begin{subfigure}[b]{0.5\textwidth}
    \includegraphics[width=\linewidth]{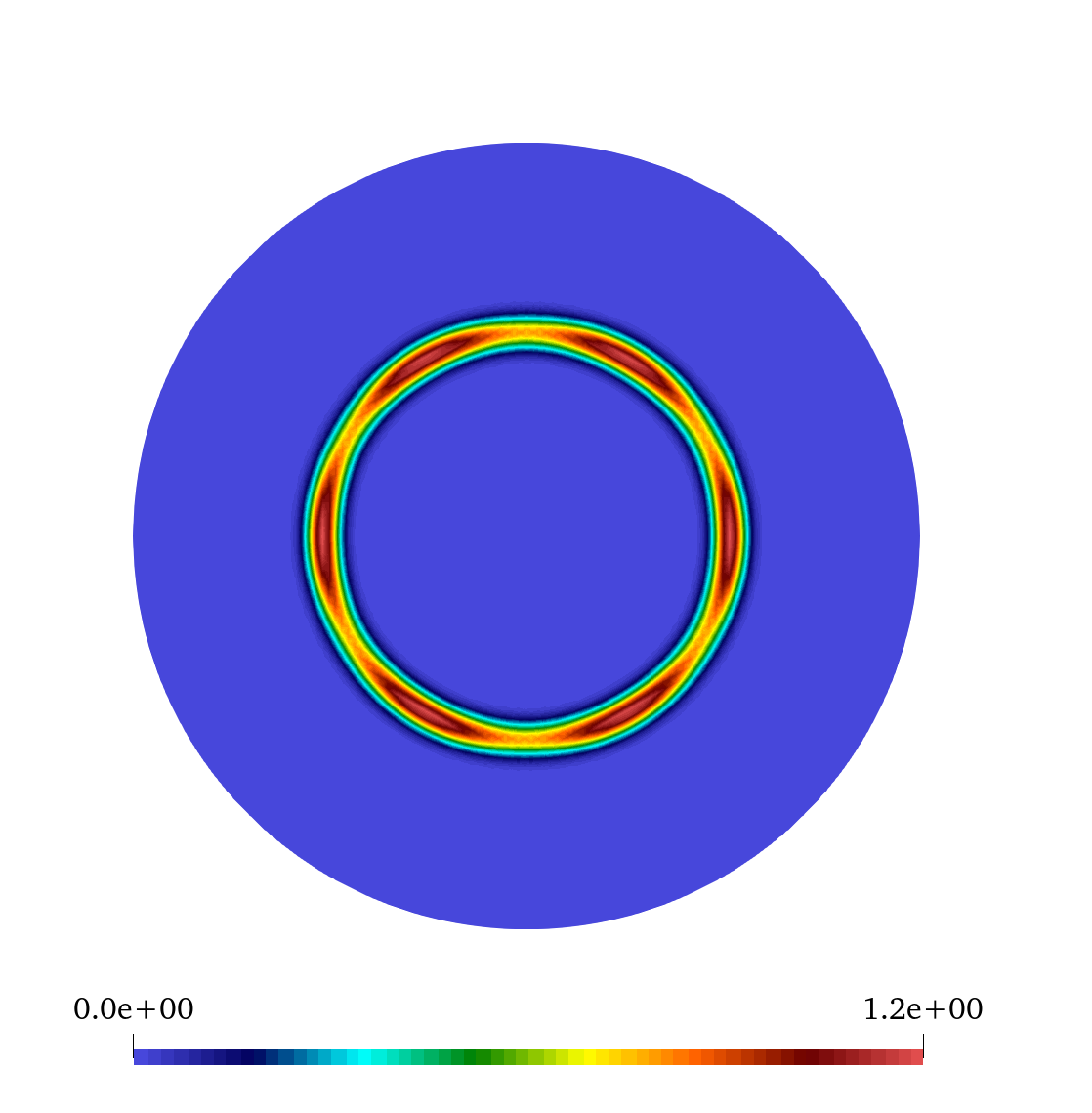}
    \caption{$t=0$}
  \end{subfigure}\hfill
  \centering
  \begin{subfigure}[b]{0.5\textwidth}
    \includegraphics[width=\linewidth]{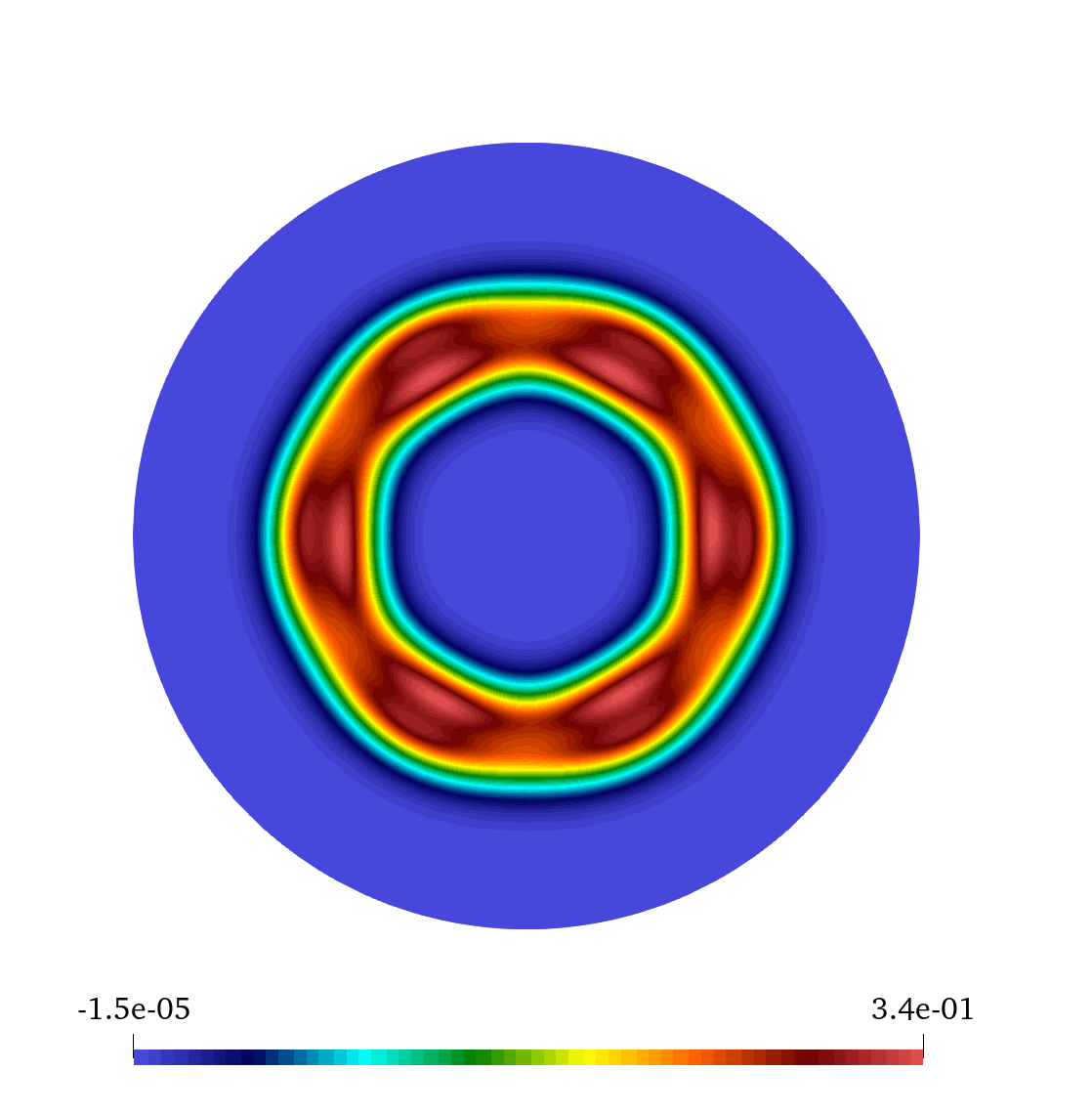}
    \caption{$t=1$}
  \end{subfigure}\hfill
  \caption{Diocotron instability: time evolution of the charge density. The solutions are obtained in such mesh: $\Omega_{\pmb{x}}$ is discretized using $\polP_1$ element with $10000$ DOFs in total, and $\Omega_{\pmb{v}}$ is discretized using $\polQ_1$ element with $20\times20$ DOFs.}
  \label{fig:DI1}
\end{figure}

\begin{figure}[htbp]
  \centering
  \begin{subfigure}[b]{0.5\textwidth}
    \includegraphics[width=\linewidth]{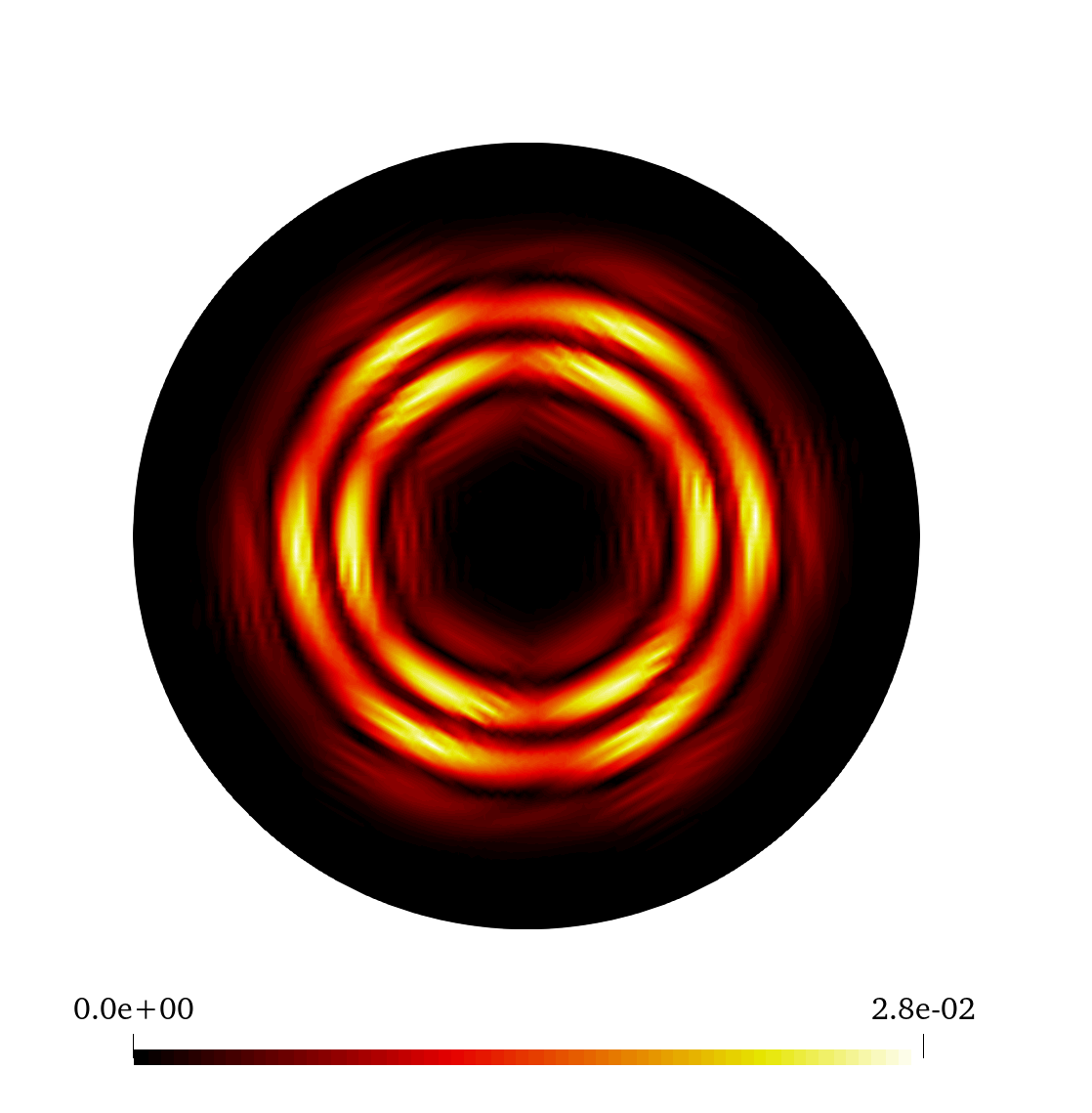}
    \caption{$R_x$}
  \end{subfigure}\hfill
  \centering
  \begin{subfigure}[b]{0.5\textwidth}
    \includegraphics[width=\linewidth]{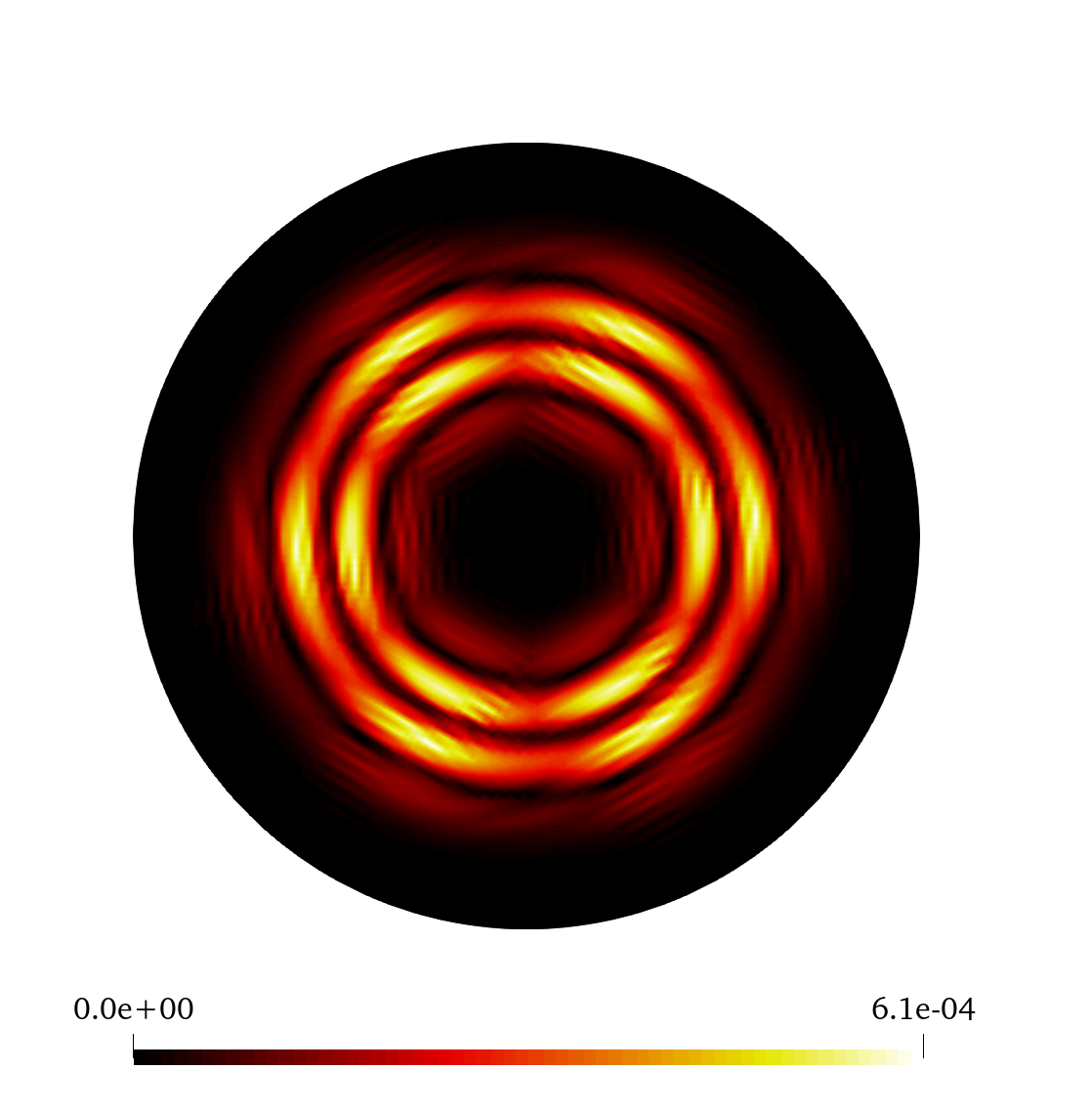}
    \caption{$\mathsf{A}_\bx$}
  \end{subfigure}\hfill
  \caption{Diocotron instability: the residual $R_x$ and the artificial viscosity $\mathsf{A}_{\pmb{x}}$ at $t=1$.}
  \label{fig:DI2}
\end{figure}

\section{Conclusion and outlook}
In this paper, we proposed and numerically solved a viscous regularization of the Vlasov--Maxwell system. The viscous regularization enhances the finite element approximation for convection-dominated problems, such as the Vlasov equation. We then presented a high-order finite element method to solve this system. The method employs high-order polynomials on general cuboids in $\polR^{d_{\pmb{x}}}$ and $\polR^{d_{\pmb{v}}}$. The Vlasov equation is discretized using Lagrange polynomial spaces, while the Maxwell system is approximated using curl-and divergence-conforming finite element spaces. Specifically, we use Nedelec and Raviart--Thomas spaces. The method utilizes tensor-product constructions, enabling the application to high-dimensional problems, i.e., when $d_{\pmb{x}} + d_{\pmb{v}} > 3$. An additional advantage of using tensor products is the increased flexibility in the choice of finite elements. The numerical experiments were conducted for the reduced Vlasov--Maxwell system. The results demonstrate that the proposed method is robust, highly accurate, mass-conservative, and preserves Gauss' law. As a future direction, we aim to enhance the viscous regularization term to ensure that the method is provably positivity-preserving.

\section*{Acknowledgements}
The second and third authors are financially supported by the Swedish Research Council (VR) under grant numbers 2021-04620 and 2021-05095, which is gratefully acknowledged.  Some computations were performed on UPPMAX provided by the National Academic Infrastructure for Super­computing in Sweden (NAISS) under project NAISS 2023/23-12. Open access funding is provided by Uppsala University. The first author is supported by the German Science Foundation DFG within the Collaborative Research Center SFB1491.

\bibliographystyle{abbrvnat} 
\bibliography{ref}

\appendix

\end{document}